\pdfoutput=1 
\documentclass[11pt]{amsart}

\usepackage{epsfig,overpic,comment}
\usepackage[usenames,dvipsnames,svgnames,table]{xcolor}
\usepackage[hyphens]{url}
\usepackage[pagebackref,linktocpage=true,colorlinks=true,linkcolor=Blue,citecolor=BrickRed,urlcolor=RoyalBlue]{hyperref}
\usepackage[msc-links,abbrev]{amsrefs}
\usepackage{amsmath,amsthm,amssymb,marginnote}
\usepackage{enumerate}
\usepackage[T1]{fontenc}

\newtheorem{theorem}{Theorem}[section]
\newtheorem{corollary}[theorem]{Corollary}
\newtheorem{lemma}[theorem]{Lemma}
\newtheorem{proposition}[theorem]{Proposition}
\newtheorem{problem}[theorem]{Problem}

\theoremstyle{definition}
\newtheorem{note}[theorem]{Note}

\newcommand{\be}{\begin{equation}}
\newcommand{\ee}{\end{equation}}
\newcommand{\ol}{\overline}
\newcommand{\goto}{\rightarrow}
\newcommand{\R}{\mathbf{R}}

\newcommand{\e}{\varepsilon}
\newcommand{\C}{\mathcal{C}}

\renewcommand{\epsilon}{\varepsilon}
\renewcommand{\S}{\mathbf{S}}
\renewcommand{\tilde}{\widetilde}

\DeclareMathOperator{\cone}{cone}
\DeclareMathOperator{\conv}{conv}
\DeclareMathOperator{\cl}{cl}

\DeclareMathOperator{\vol}{vol}

\makeatletter
\DeclareFontFamily{U}{tipa}{}
\DeclareFontShape{U}{tipa}{m}{n}{<->tipa10}{}
\newcommand{\arc@char}{{\usefont{U}{tipa}{m}{n}\symbol{62}}}%

\newcommand{\arc}[1]{\mathpalette\arc@arc{#1}}

\newcommand{\arc@arc}[2]{%
  \sbox0{$\m@th#1#2$}%
  \vbox{
    \hbox{\resizebox{\wd0}{\height}{\arc@char}}
    \nointerlineskip
    \box0
  }%
}
\makeatother

\makeatletter
\def\@tocline#1#2#3#4#5#6#7{\relax
  \ifnum #1>\c@tocdepth 
  \else
    \par \addpenalty\@secpenalty\addvspace{#2}%
    \begingroup \hyphenpenalty\@M
    \@ifempty{#4}{%
      \@tempdima\csname r@tocindent\number#1\endcsname\relax
    }{%
      \@tempdima#4\relax
    }%
    \parindent\z@ \leftskip#3\relax \advance\leftskip\@tempdima\relax
    \rightskip\@pnumwidth plus4em \parfillskip-\@pnumwidth
    #5\leavevmode\hskip-\@tempdima
      \ifcase #1
       \or\or \hskip 1.3em \or \hskip 2em \else \hskip 5em \fi%
      #6\nobreak\relax
    \hfill\hbox to\@pnumwidth{\@tocpagenum{#7}}\par
    \nobreak
    \endgroup
  \fi}
\makeatother

\newcommand{\nocontentsline}[3]{}
\newcommand{\tocless}[2]{\bgroup\let\addcontentsline=\nocontentsline#1{#2}\egroup}

\begin{document}
\setlength{\baselineskip}{1.2\baselineskip}

\title[Total curvature and the isoperimetric inequality] 
{Total curvature and the isoperimetric inequality \\in Cartan-Hadamard manifolds}

\author{Mohammad Ghomi}
\address{School of Mathematics, Georgia Institute of Technology,
Atlanta, GA 30332}
\email{ghomi@math.gatech.edu}
\urladdr{www.math.gatech.edu/~ghomi}

\author{Joel Spruck}
\address{Department of Mathematics, Johns Hopkins University,
 Baltimore, MD 21218}
\email{js@math.jhu.edu}
\urladdr{www.math.jhu.edu/~js}

\vspace*{-0.75in}
\begin{abstract} 
We obtain an explicit  formula for comparing total curvature of level sets of functions on Riemannian manifolds, and develop some applications of this result to the isoperimetric  problem  in spaces of nonpositive curvature.
\end{abstract}

\date{\today \,(Last Typeset)}
\subjclass[2010]{Primary: 53C20, 58J05; Secondary: 52A38, 49Q15.}
\keywords{Signed Distance Function, Cut Locus,  Medial Axis, Positive Reach,  Inf-Convolution,    Quermassintegrals, Semiconcave Functions,  Convex Hull,    Isoperimetric Profile.}
\thanks{The research of M.G. was supported in part by NSF grant DMS-1711400 and a Simons fellowship.}

\maketitle
\vspace*{-0.4in}
\tableofcontents

\section{Introduction}\label{sec1}

A \emph{Cartan-Hadamard} manifold $M$ is a complete simply connected Riemannian space with nonpositive (sectional) curvature.  
 In this paper we are concerned with generalizing some fundamental properties of Euclidean space $\R^n$ to Cartan-Hadamard manifolds.
The first problem involves \emph{convex} subsets of $M$, i.e., those which contain the geodesic connecting every pair of their points.  A \emph{(closed) convex hypersurface} $\Gamma\subset M$ is the boundary of a compact convex set with interior points. If $\Gamma$ is of regularity class $\C^{1,1}$, then its \emph{Gauss-Kronecker curvature}, or determinant of second fundamental form, $GK$ is well-defined almost everywhere. So the \emph{total curvature} of $\Gamma$ may be defined as $\mathcal{G}(\Gamma):=\int_\Gamma |GK| d\sigma$, where $d\sigma$ is the volume form of $\Gamma$. When $M=\R^n$, $\mathcal{G}(\Gamma)$ is the  volume of the Gauss map, or a unit normal vector field on $\Gamma$. Thus it is easy to see that
\be\label{eq:GK}
\mathcal{G}(\Gamma)\geq \textup{vol}(\S^{n-1}),
\ee 
where 
$\S^{n-1}$ denotes the unit sphere in $\R^n$ and \emph{vol} stands for volume. 

\begin{problem}\label{prob:1}
Does the total curvature inequality \eqref{eq:GK} hold for convex $\C^{1,1}$ hypersurfaces in Cartan-Hadamard manifolds $M^n$?
\end{problem}

When $n=2$, $3$ the answer is yes due to Gauss-Bonnet theorem and Gauss' equation, but for $n\geq 4$ the problem is open. Work on this question goes back at least to 1966, when Willmore and Saleemi \cite{willmore-saleemi} claimed incorrectly to have found a simple proof. 
Several authors  have studied the problem since then, e.g., see \cite{dekster1981, dekster1982, schroeder-strake, borbely2002, borbely2003, cao-escobar}, and it has been explicitly posed by Gromov \cite[p. 66]{ballmann-gromov-schroeder}.
 A prime motivation for studying Problem \ref{prob:1} is its connection to the classical isoperimetric inequality \cites{osserman1978,  bandle2017, chavel2001, hhm1999, blasjo2005}, which states that for any bounded set $\Omega\subset\R^n$,
\be\label{eq:II}
 \textup{per}(\Omega)^n  \geq 
\frac{\textup{per}(\textbf{B}^{n})^n }{\;\;\;\;\textup{vol}(\mathbf{B}^n)^{n-1}}\,\textup{vol}(\Omega)^{n-1},
\ee
 where $\emph{per}$ stands for perimeter, and $\textbf{B}^{n}$ is the unit ball in $\R^n$ (so $\textup{per}(\textbf{B}^{n})=\textup{vol}(\textbf{S}^{n-1})$). Furthermore,  equality holds only if $\Omega$ is a ball in $\R^n$. 
  
 \begin{problem}\label{prob:2}
Does the isoperimetric inequality \eqref{eq:II} hold for bounded sets in Cartan-Hadamard manifolds $M^n$?
\end{problem}

The assertion that the answer is yes has become known as the \emph{Cartan-Hadamard conjecture} \cites{kloeckner-kuperberg2017, aubin-druet-hebey1998, druet-hebey2002, hebey-robert2008, kristaly2019}.
 Weil \cite{weil1926} \cite[p. 347]{berger2003} established the conjecture for $n=2$ in 1926, and Beckenbach-Rado \cite{beckenbach-rado1933} rediscovered the same result in 1933. 
 In 1975 Aubin \cite{aubin1975}  stated the conjecture for $n\geq 3$,
 as did Gromov \cites{gromov1981,gromov:metric}, and Burago-Zalgaller \cites{burago-zalgaller1980}\cite[Sec. 36.5.10]{burago-zalgaller1988} a few years later. Subsequently the cases $n=3$ and $4$ of the conjecture were established,  by Kleiner \cite{kleiner1992} in 1992, and Croke \cite{croke1984} in 1984 respectively, using different methods.  See   
 \cite[Sec. 3.3.2]{ritore2010} and \cite{schulze2008} for alternative proofs for $n=3$, and \cite{kloeckner-kuperberg2017} for another proof for $n=4$. Other related studies and references may be found in
 \cites{druet2010, morgan-johnson2000,  druet2002, hass2016,  schulze2018,nardulli-acevedo2018, cao-escobar}.

This paper is motivated by the work of Kleiner \cite{kleiner1992}, who showed that when $n=3$ the total curvature inequality \eqref{eq:GK} implies the isoperimetric inequality \eqref{eq:II}, and stated that this implication should hold in all dimensions.  Here we show that this is indeed the case (Theorem \ref{thm:GK-CH}). So an affirmative answer to Problem \ref{prob:1} yields an affirmative answer to Problem \ref{prob:2}. The rest of this paper is 
devoted to developing a number of tools and techniques for solving Problem \ref{prob:1}. Foremost among these is a comparison formula for total curvature of level sets of functions on Riemannian manifolds (Theorems \ref{thm:comparison} and \ref{thm:comparison2}). We also establish some results on the structure of the convex hulls, and regularity properties and singularities of the distance function which may be of independent interest.

We start in Section \ref{sec:distance} by recording a number of basic facts concerning regularity properties of distance functions on Riemannian manifolds which will be useful throughout the paper. In Section \ref{sec:d-convex} we discuss various notions of convexity in a Cartan-Hadamard manifold, and show that it is enough to establish the total curvature inequality \eqref{eq:GK} for hypersurfaces with convex distance function. In Section \ref{sec:comparison} we establish the comparison formula, and discuss its applications in Section \ref{sec:applications} for total curvature of nested hypersurfaces in the hyperbolic space, and parallel hypersurfaces in any Riemannian manifold. In Section \ref{sec:convexhull} we will show that the total positive curvature of the convex hull of a hypersurface cannot exceed that of the hypersurface itself. This result will be used in Section \ref{sec:CHC} to establish the connection between problems \ref{prob:1} and \ref{prob:2}. 

A basic approach to solving Problem \ref{prob:1} would be to  shrink the convex hypersurface $\Gamma$, without increasing $\mathcal{G}(\Gamma)$, until it collapses to a point. Since all Riemannian manifolds are locally Euclidean up to first order, we would then obtain \eqref{eq:GK} in the limit. Our comparison formula shows that $\mathcal{G}(\Gamma)$ does not increase when $\Gamma$ is moved parallel to itself inward until we reach the first singularity of the distance function or cut locus of $\Gamma$ (Corollary \ref{cor:parallel}). It might be possible to extend this result further via an appropriate smoothing of the distance function. To this end we gather some relevant results in Appendices \ref{sec:inf-convolution} and \ref{sec:projection}. There is a more conventional geometric flow, by harmonic mean curvature, which also shrinks convex hypersurfaces  to a point in Cartan-Hadamard manifolds \cite{xu2010}; however, it is not known how that  flow effects total curvature; see also \cite{andrews-hu-li}.

The isoperimetric inequality has several well-known applications  \cites{bandle1980, polya-szego1951, payne1967} due to its relations with many other important inequalities \cites{treibergs2002,chavel2001}. For instance a positive resolution of Cartan-Hadamard conjecture would extend the classical Sobolev inequality from the Euclidean space to Cartan-Hadamard manifolds  \cite{osserman1978,federer-fleming1960,druet2010}, \cite[App.1]{milman1986}. Indeed it was in this context where  the Cartan-Hadamard conjecture was first proposed \cite{aubin1975}.
See \cite{druet-hebey2002,kristaly2015, kristaly2019} for a host of other Sobolev type inequalities on Cartan-Hadamard manifolds which would follow from the conjecture, and \cites{yau1975, hoffman-spruck1974, chung-grigoryan-yau2000} for related studies. The isoperimetric inequality also has deep connections to spectral analysis. A fundamental result in this area is the Faber-Krahn inequality \cites{chavel1984,henrot2006,berger2003} which was established in 1920s \cite{faber1923, krahn1925, krahn1926} in Euclidean space, as had been conjectured by Rayleigh in 1877 \cite{rayleigh1877}. This inequality may also be generalized to Cartan-Hadamard manifolds \cite{chavel1984}, if Cartan-Hadamard conjecture holds. 

We should mention that the Cartan-Hadamard conjecture has a stronger form \cites{aubin1975,gromov:metric,burago-zalgaller1988}, sometimes called the \emph{generalized Cartan-Hadamard conjecture} \cite{kloeckner-kuperberg2017}, which states that if the sectional curvatures of $M$ are bounded above by $k\leq 0$, then the perimeter of $\Omega$ cannot be smaller than that of a ball of the same volume in the hyperbolic space of constant curvature $k$.  The generalized conjecture has been proven only for $n=2$ by Bol \cite{bol1941}, and $n=3$ by Kleiner \cite{kleiner1992}; see also Kloeckner-Kuperberg \cite{kloeckner-kuperberg2017} for partial results for $n=4$, and Johnson-Morgan \cite{morgan-johnson2000} for a result on small volumes.

\section{Regularity and Singular Points of the Distance Function}\label{sec:distance}
Throughout this paper, $M$ denotes a complete connected Riemannian manifold of dimension $n\geq 2$ with metric $\langle\cdot,\cdot\rangle$ and corresponding distance function $d\colon M\times M\to \R$.  For any pairs of sets $X$, $Y\subset M$, we define 
$$
d(X,Y):=\inf \{d(x,y)\mid x\in X, \;y\in Y\}.
$$
Furthermore, for any set $X\subset M$, we define $d_X\colon M\to \R$, by
$$
d_X(\,\cdot\,):=d(X,\,\cdot\,).
$$
The \emph{tubular neighborhood} of $X$ with radius $r$ is then given by $U_r(X):=d_X^{-1}([0,r))$. Furthermore, for any $t>0$, the level set $d_X^{-1}(t)$ will be called a \emph{parallel hypersurface} of $X$ at distance $t$.
A function $u\colon M\to\R$ is  Lipschitz with constant $L$, or \emph{$L$-{Lipschitz}}, if for all pairs of points $x$, $y\in M$, $|u(x)-u(y)|\leq L\,d(x,y)$.
The triangle inequality and Rademacher's theorem quickly yield \cite[p.185]{dugundji1966}:
\begin{lemma}[\cite{dugundji1966}]\label{lem:lipschitz}
For any set $X\subset M$, $d_X$ is $1$-Lipschitz. In particular  $d_X$ is differentiable almost everywhere.
\end{lemma}
For any point $p\in M$ and $X\subset M$, we say that $p^{\circ}\in X$ is a  \emph{footprint} of $p$ on $X$ provided that 
$$
d(p,p^\circ)=d_X(p),
$$
 and the distance minimizing geodesic connecting $p$ and $p^\circ$ is unique. In particular note that every point of $X$ is its own footprint. The following observation is well-known when $M=\R^n$. It follows, for instance, from studying super gradients of semiconcave functions \cite[Prop. 3.3.4 \& 4.4.1]{cannarsa-sinestrari2004}.
These arguments extend well to Riemannian manifolds \cite[Prop. 2.9]{mantegazza-mennucci2003}, since local charts preserve both semiconcavity and generalized derivatives.  For any function $u\colon M\to \R$, we let $\nabla u$ denote its gradient.

\begin{lemma}[\cites{cannarsa-sinestrari2004,mantegazza-mennucci2003}] \label{lem:CS}
Let $X\subset M$ be a closed set, and $p\in M\setminus X$. Then
\begin{enumerate}[(i)]
\item{$d_X$ is differentiable at $p$ if and only if $p$ has a unique footprint on $X$.}
\item{If $d_X$ is differentiable at $p$, then $\nabla d_X(p)$ is tangent to the distance minimizing geodesic connecting $p$ to its footprint on $X$, and $|\nabla d_X(p)|=1$.}
\item{$d_X$ is $\C^1$ on any open set in $M\setminus X$ where $d_X$ is pointwise differentiable.}
\end{enumerate}
\end{lemma}

Throughout this paper, $\Gamma$ will denote a closed embedded topological hypersurface in $M$. Furthermore we assume that $\Gamma$ bounds a designated \emph{domain}  $\Omega$,  i.e., a connected open set with compact closure $\cl(\Omega)$ and boundary
$$
\partial\Omega=\Gamma.
$$
The  \emph{(signed) distance function} $\widehat{d}_\Gamma\colon M\to\R$ of $\Gamma$ (with respect to $\Omega$) is then given by
$$
\widehat{d}_\Gamma(\,\cdot\,):=d_\Omega(\,\cdot\,)-d_{M\setminus\Omega}(\,\cdot\,).
$$
In other words, $\widehat{d}_\Gamma(p)=-d_\Gamma(p)$ if $p\in \Omega$, and $\widehat{d}_\Gamma(p)=d_\Gamma(p)$ otherwise. The level sets $\widehat{d}_\Gamma^{-1}(t)$ will be called \emph{outer parallel hypersurfaces} of $\Gamma$ if $t>0$, and \emph{inner parallel hypersurfaces} if $t<0$.
 Let $\textup{reg}(\widehat{d}_\Gamma)$ be the union of all open sets in $M$ where each point has a unique footprint on $\Gamma$. 
Then the  \emph{cut locus} of $\Gamma$  is defined as 
$$
\textup{cut}(\Gamma):=M\setminus \textup{reg}\big(\widehat{d}_\Gamma\big).
$$
For instance when $\Gamma$ is an ellipse in $\R^2$, $\textup{cut}(\Gamma)$ is the line segment in $\Omega$ connecting the focal points of the inward normals (or the cusps of the evolute) of $\Gamma$, see Figure \ref{fig:ellipse}.
\begin{figure}[h]
\centering
\begin{overpic}[height=1in]{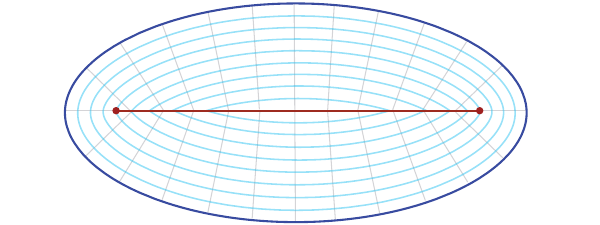}
\put( 14,4 ){\Small $\Gamma$}
\put( 62,15 ){\Small $\text{cut}(\Gamma)$}
\put( 47.5,5 ){\Small $\Omega$}
\end{overpic}
\caption{}\label{fig:ellipse}
\end{figure}
Note that the singularities of parallel hypersurfaces of $\Gamma$ all lie on $\text{cut}(\Gamma)$. Since
  $d_\Gamma$ may not be differentiable at any point of $\Gamma$,  we find it more convenient to work with $\widehat{d}_\Gamma$ instead. Part (iii) of Lemma \ref{lem:CS} may be extended as follows:
\begin{lemma}\label{lem:widehatdC1}
If $\Gamma$ is $\C^1$, then $\widehat{d}_\Gamma$ is $\C^1$ on $M\setminus\textup{cut}(\Gamma)$ with $|\nabla\widehat{d}_\Gamma|=1$.
\end{lemma}
\begin{proof}
By Lemma \ref{lem:CS}, $\widehat{d}_\Gamma$ is $\C^1$ on 
$(M\setminus\Gamma)\setminus\textup{cut}(\Gamma)$. Thus it remains to consider the regularity 
of $\widehat{d}_\Gamma$ on $\Gamma\setminus\textup{cut}(\Gamma)$. To this end let $p\in \Gamma\setminus\textup{cut}(\Gamma)$, and $U$ be a convex open neighborhood of $p$ in $M$ which is disjoint 
from $\textup{cut}(\Gamma)$. Then each point of $U$ has a unique footprint on $\Gamma\cap U$, and thus $U$ is fibrated by geodesic segments orthogonal to $\Gamma\cap U$. For convenience,  we may assume that all these segments have the same length. 
Now let $\Gamma_\e:=(\widehat{d}_\Gamma)^{-1}(\e)$ where $\e>0$ is so small that $\Gamma_\e$ intersects $U$. Then each point of $\Gamma_\e\cap U$ has a unique foot print on $\Gamma\cap U$. Furthermore, by  Lemma \ref{lem:CS}, $\Gamma_\e\cap U$ is a $\C^1$ hypersurface, since $\widehat{d}_\Gamma$ is $\C^1$ on $U\setminus\Gamma$ and has nonvanishing gradient there. So $\Gamma_\e\cap U$ is orthogonal to the geodesic segments fibrating $U$. Since these segments do not intersect each other,  $U$ is disjoint from $\textup{cut}(\Gamma_\e)$. So
$\widehat{d}_{\Gamma_\e}$ is $\C^1$ on $U\setminus\Gamma_\e$ by Lemma \ref{lem:CS}. Finally note that 
$\widehat{d}_\Gamma=\widehat{d}_{\Gamma_\e}-\e$ on $U$, which completes the proof.
\end{proof}
The \emph{medial axis} of $\Gamma$, $\textup{medial}(\Gamma)$, is  the set of points in $M$ with multiple footprints on $\Gamma$. Note that
\be\label{eq:medial}
\textup{cut}(\Gamma)=\cl\big(\text{medial}(\Gamma)\big).
\ee
For instance, when $\Gamma$ is an ellipse in $\R^2$, $\textup{medial}(\Gamma)$ is the relative interior of the segment connecting its foci.
Let $\textup{sing}(\widehat{d}_\Gamma)$ denote the set of \emph{singularities} of $\widehat{d}_\Gamma$ or points of $M$ where $\widehat{d}_\Gamma$ is not differentiable. Then 
$$
\textup{medial}(\Gamma)=\textup{sing}\big(\widehat{d}_\Gamma\big),\quad\quad\text{and}\quad\quad\textup{cut}(\Gamma)=\cl\big(\textup{sing}(\widehat{d}_\Gamma)\big).
$$
By a (geodesic) \emph{sphere} $S\subset M$ we mean the boundary of a \emph{geodesic ball}, i.e., the image under the exponential map $\exp_p\colon T_p M\to M$ of a ball centered at $p$. 
We say that a sphere $S$ lies in a set $A\subset M$ if the ball which it bounds lies in $A$. Furthermore, $S$ is \emph{maximal} in $A$ if it is not contained in (the geodesic ball bounded by) a sphere of larger radius in $A$. The set of centers of maximal spheres contained in $\cl(\Omega)$ is called the \emph{skeleton} of $\Omega$. 

\begin{lemma}\label{lem:skeleton}
Suppose that every pair of points of $\Omega$ is connected by a unique geodesic in $M$. Then
$$
\textup{medial}(\Gamma)\cap\Omega\;\;\subset\;\; \textup{skeleton}(\Omega)\;\;\subset\; \;\cl\big(\textup{medial}(\Gamma)\cap\Omega\big).
$$
\end{lemma}
\begin{proof}
The see the first inclusion, let $x\in \textup{medial}(\Gamma)\cap\Omega$. Then there exists a sphere $S\subset \cl(\Omega)$, centered at $x$, which touches $\Gamma$ in (at least) two distinct points, say $y$ and $y'$.  Suppose there exists a sphere $S'\subset \cl(\Omega)$ which contains $S$. Then $y$, $y'$ lie on $S'$. Consider the geodesic rays $\gamma$, $\gamma'$ in $\cl(\Omega)$ which start at $y$, $y'$ and are orthogonal to $S$. These geodesics meet for the first time at $x$, due to the uniqueness property of geodesics in $\Omega$. But $\gamma$, $\gamma'$ are also orthogonal to $S'$. Thus $x$ is the center of $S'$ as well. Hence $S'=S$, which means that $S$ is maximal. So $x\in \textup{skeleton}(\Omega)$.

To see the second inclusion, let $x\in \textup{skeleton}(\Omega)$. Then there exists a maximal sphere $S$ in $\cl(\Omega)$ centered at $x$. By \eqref{eq:medial}, it suffices to show that $x\in\textup{cut}(\Gamma)$. Suppose that $x\not\in\textup{cut}(\Gamma)$. Then, by Lemma \ref{lem:CS}, $d_\Gamma$ is $\C^1$ in a neighborhood $U$ of $x$. Furthermore $\nabla d_\Gamma$ does not vanish on $U$, and its integral curves are distance minimizing geodesics connecting points of $U$ to their unique footprints on $\Gamma$. It follows then that the geodesic connecting $x$ to its footprint in $\Gamma$, may be extended at $x$ to a longer distance minimizing geodesic. This contradicts the maximality of $S$ and completes the proof.
\end{proof}

Note that the hypothesis of Lemma \ref{lem:skeleton} is used only to establish the first inclusion.
The inclusion relations in this lemma are in general strict, even when $M=\R^n$ \cite{chazal-soufflet2004}.
There is a vast literature on the singularities of the distance function, due to its applications in a number of fields, including computer vision, and connections to Hamilton-Jacobi equations; see \cites{mantegazza-mennucci2003,li-nirenberg2005,ardoy2014,damon2006, mather1983} for more references and background. Lemma \ref{lem:widehatdC1} may be extended as follows:

\begin{lemma}[\cites{foote1984,mantegazza-mennucci2003}]\label{lem:MM}
For $k\geq 2$, if $\Gamma$ is $\C^k$, then $\widehat{d}_\Gamma$ is $\C^k$ on $M\setminus \textup{cut}(\Gamma)$.
\end{lemma}

\noindent This fact has been well-known for $M=\R^n$ and $k\geq 2$, as it follows from the basic properties of the normal bundle of $M$,  and applying the  inverse function theorem to the exponential map, e.g.  see \cite{foote1984} or \cite[Sec. 2.4]{ghomi2001}.  For Riemannian manifolds, the lemma has been  established in 
\cite[Prop. 4.3]{mantegazza-mennucci2003}, via essentially the same exponential mapping argument in \cite{foote1984}.

For the purposes of this work, we still need to gather finer information about Lipschitz regularity of derivatives of $\widehat{d}_\Gamma$. To this end we invoke Federer's notion of \emph{reach}  \cite{federer1959, thale2008} which may be defined as
$$
\textup{reach}(\Gamma):=d\big(\Gamma, \textup{cut}(\Gamma)\big).
$$
In particular note that $\textup{reach}(\Gamma)\geq r$ if and only if there exists  a geodesic ball of radius $r$ \emph{rolling freely} on each side of $\Gamma$ in $M$, i.e., through each point $p$ of $\Gamma$ there passes the boundaries of geodesic balls $B$, $B'$ of radius $r$ such that $B\subset\cl(\Omega)$, and $B'\subset M\setminus\Omega$. We say that $\Gamma$ is $\C^{1,1}$, if it is $\C^{1,1}$ in local charts, i.e., for each point $p\in M$ there exists a neighborhood $U$ of $p$ in $M$, and a $\C^\infty$ diffeomorphism $\phi\colon U\to\R^n$ such that $\phi(\Gamma)$ is $\C^{1,1}$ in $\R^n$. A function $u\colon M\to\R$ is called \emph{locally $\C^{1,1}$} on some region $X$, if it is $\C^{1,1}$ in local charts covering $X$. If $X$ is compact, then we simply say that $u$ is $\C^{1,1}$ near $X$.
\begin{lemma}\label{lem:GH}
The following conditions are equivalent:
\begin{enumerate}[(i)]
\item{$\textup{reach}(\Gamma)>0$.}
\item{$\Gamma$ is  $\C^{1,1}$.} 
\item{$\widehat{d}_\Gamma$ is  $\C^{1,1}$ near $\Gamma$.}
\end{enumerate}
\end{lemma}
\begin{proof}
For $M=\R^n$, the equivalence (i)$\Leftrightarrow$(ii)
is due to \cite[Thm. 1.2]{ghomi-howard2014}, since $\Gamma$ is a topological hypersurface by assumption, and the positiveness of reach, or more specifically existence of local support balls on each side of $\Gamma$, ensures that the tangent cones of $\Gamma$ are all flat. The general case then may be reduced to the Euclidean one via local charts. Indeed local charts of $M$ preserve the $\C^{1,1}$ regularity of $\Gamma$ by definition. Furthermore, the positiveness of reach is also preserved, as shown in the next paragraph; see also \cite{bangert1982}. 

Let $(U, \phi)$ be a local chart of $M$ around a point $p$ of $\Gamma$. We may assume that $\phi(U)$ is a ball $B$ in $\R^n$. Furthermore, since $\Gamma$ is a topological hypersurface, we may assume that $\Gamma$ divides $U$ into a pair of components by the  Jordan Brouwer separation theorem. Consequently $\phi(\Gamma\cap U)$ divides $B$ into a pair of components as well, which we call the sides of $\phi(\Gamma)$. The image under $\phi$ of the boundary of the balls of some constant radius which roll freely on each side of $\Gamma$ in $M$ generate closed $\C^2$ surfaces $S_x$, $S'_x$ on each side of every point $x$ of $\phi(\Gamma)$. Let $B'\subset B$ be a smaller ball centered at $\phi(p)$, and 
$X$ be the connected component of $\phi(\Gamma\cap U)$  in $B'$ which contains $\phi(p)$. Furthermore, let
$\kappa$ be the supremum of the principal curvatures of 
$S_x$, $S'_x$, for all $x\in X$. Then $\kappa<\infty$, since $X$ has compact closure in $B$ and the principal curvatures of $S_x$, $S'_x$ vary continuously, owing to the fact that $\phi$ is $\C^2$. It is not difficult then to show that the reach of $S_x$, $S'_x$ is uniformly bounded below, which will complete the proof. Alternatively, we may let $(U, \phi)$ be a normal coordinate chart generated by the exponential map. Then for $U$ sufficiently small, $S_x$ and $S'_x$ will have positive principal curvatures. So, by Blaschke's rolling theorem \cite{blaschke:Kreis,brooks-strantzen1989}, a ball  rolls freely inside $S_x$, $S_x'$ and consequently on each 
side of $\phi(\Gamma\cap U)$ near $\phi(p)$. Hence $\phi(\Gamma\cap U)$ has positive reach near $\phi(p)$, as desired.

It remains then to establish the equivalence of (iii) with (i) or (ii). First suppose that (iii) holds. Let $p\in\Gamma$ and $U$ be neighborhood of $p$ in $M$ such that $u:=\widehat{d}_\Gamma$ is $\C^{1,1}$ on $U$. 
By Lemma \ref{lem:CS}, $|\nabla u|\equiv1$ on $U\setminus\Gamma$. Also note that each point of $\Gamma$ is a limit of points of $U\setminus\Gamma$, since by assumption $\Gamma$ is a topological hypersurface.
Thus, since $u$ is $\C^1$ on $U$, it follows that $|\nabla u|\neq 0$ on $U$. In particular,  
$\Gamma\cap U$ is a regular level set of $u$ on $U$, and  is $\C^1$ by the inverse function theorem. Let $\phi\colon U \to\R^n$ be a diffeomorphism. Then $\phi(\Gamma\cap U)$ is a regular level set of the locally $\C^{1,1}$ function $u\circ\phi^{-1}\colon\R^n\to\R$. In particular the unit normal vectors of $\phi(\Gamma\cap U)$ are locally Lipschitz continuous, since they are given by $\nabla(u\circ\phi^{-1})/|\nabla(u\circ\phi^{-1})|$. So $\phi(\Gamma\cap U)$ is locally $\C^{1,1}$. Hence $\Gamma$ is locally $\C^{1,1}$, and so we have established that (iii)$\Rightarrow$(ii). Conversely, suppose that (ii) and therefore (i) hold. Then any point $p\in\Gamma$ has an open neighborhood $U$ in $M$ where each point has a unique footprint on $M$. Thus, by Lemma \ref{lem:widehatdC1}, $u$ is $\C^1$ on $U$ and its gradient vector field is tangent to geodesics orthogonal to $\Gamma$. So, for $\epsilon$ small, each  level set $u^{-1}(\epsilon)\cap U$ has positive reach and is therefore $\C^{1,1}$ by (ii). Via local charts we may transfer this configuration to $\R^n$, to generate a fibration of $\R^n$ by $\C^{1,1}$ hypersurfaces which form the level sets of $u\circ\phi^{-1}$. Since $\nabla(u\circ\phi^{-1})/|\nabla(u\circ\phi^{-1})|$ is orthogonal to these level sets, it follows then that $\nabla(u\circ\phi^{-1})$ is locally Lipschitz. Thus $u\circ\phi^{-1}$ is locally $\C^{1,1}$ which establishes (iii) and completes the proof.
\end{proof}

The following proposition for $M=\R^n$ is originally due to Federer \cite[Sec. 4.20]{federer1959}; see also \cite[p. 365]{delfour-zolesio2011}, \cite[Sec. 3.6]{cannarsa-sinestrari2004}, and \cite{clarke-stern-wolenski1995}. In \cite[Rem. 4.4]{mantegazza-mennucci2003}, it is mentioned that Federer's result should hold in all Riemannian manifolds. Indeed it follows quickly from Lemma \ref{lem:GH}:

\begin{proposition}\label{prop:C11}
$\widehat{d}_\Gamma$ is  locally $\C^{1,1}$ on $M\setminus\textup{cut}(\Gamma)$. In particular if $\Gamma$ is $\C^{1,1}$, then $\widehat{d}_\Gamma$ is  locally $\C^{1,1}$ on $U_r(\Gamma)$  for $r:=\textup{reach}(\Gamma)$. 
\end{proposition}
\begin{proof}
For each point $p\in M\setminus\textup{cut}(\Gamma)$, let $\alpha_p$ be the (unit speed) geodesic in $M$ which passes through $p$ and is tangent to $\nabla \widehat{d}_\Gamma(p)$. By Lemma \ref{lem:CS}, $\alpha_p$ is a trajectory of the gradient field $\nabla \widehat{d}_\Gamma$ near $p$. It follows that  these geodesics fibrate $M\setminus\textup{cut}(\Gamma)$. Consequently the level set 
$\{\widehat{d}_\Gamma=\widehat{d}_\Gamma(p)\}$ has positive reach near $p$, since it is orthogonal to the gradient field.
 So, by Lemma \ref{lem:GH}, $\widehat{d}_\Gamma$ is $\C^{1,1}$ near $p$, which completes the proof. 
 \end{proof}
 
 We will also need the following refinement of Proposition \ref{prop:C11}, which gives an estimate for the $\C^{1,1}$ norm of $\widehat{d}_\Gamma$ near $\Gamma$, depending only on $\text{reach}(\Gamma)$ and the sectional curvature $K_M$ of $M$; see also Lemma \ref{lem:convex-concave} below. Here $\nabla^2$ denotes the Hessian.
 
 \begin{proposition}\label{prop:C11-2}
Suppose that $r:=\textup{reach}(\Gamma)>0$, and $K_M \geq -C$, for $C\geq 0,$ on $U_r(\Gamma)$.
 Then, for $\delta:=r/2$, 
 $$
\big |\nabla^2 \widehat{d}_\Gamma\big| \leq \sqrt{C}\coth{(\sqrt{C} \delta)} $$
 almost everywhere on $U_{\delta}(\Gamma)$.
 \end{proposition}
 \begin{proof}
 By Proposition \ref{prop:C11} and Rademacher's theorem, $\widehat{d}_\Gamma$ is twice differentiable at almost every point of $U_{\delta}(\Gamma)$. Let  $p\in U_{\delta}(\Gamma)$ be such a point. Then  the eigenvalues of $\nabla^2 \widehat{d}_\Gamma(p)$, except for the one in the direction of $\nabla \widehat{d}_\Gamma(p)$ which vanishes,  are the principal curvatures of the level set $\Gamma_p:=\{\widehat{d}_\Gamma=\widehat{d}_\Gamma(p)\}$. Since by assumption a ball of radius $r$ rolls freely on each side of $\Gamma$, it follows that a ball of radius $\delta$ rolls freely on each side of $\Gamma_p$. Thus the principal curvatures of $\Gamma_p$ at $p$ are bounded above by those of spheres of radius $\delta$ in $U_r(\Gamma)$, which are in turn bounded above by $\sqrt{C}\coth{(\sqrt{C} \delta)}$ due to basic Riemannian comparison theory \cite[p. 184]{karcher1989}.
\end{proof}

\section{Notions of Convexity in Cartan-Hadamard Manifolds}\label{sec:d-convex}
A set $X\subset M$ is  \emph{(geodesically) convex} provided that every pair of its points may be joined by a unique geodesic in $M$, and that geodesic is contained in $X$. Furthermore, $X$ is \emph{strictly convex} if $\partial X$ contains no geodesic segments. In this work, a \emph{convex hypersurface}  is the boundary of a compact convex subset of $M$ with nonempty interior.  In particular $\Gamma$ is  convex if $\Omega$ is convex.
A function $u\colon M\to\R$ is \emph{convex}  provided that its composition with parameterized geodesics in $M$ is convex, i.e., for every  geodesic $\alpha\colon[t_0,t_1]\to M$,
$$
u\circ\alpha\big((1-\lambda) t_0+\lambda t_1\big)\leq (1-\lambda)\, u\circ\alpha(t_0)+ \lambda \,u\circ\alpha(t_1),
$$
for all $\lambda \in [0,1]$. We assume that all parameterized geodesics in this work have unit speed. We say that $u$ is \emph{strictly convex} if the above inequality is always strict. Furthermore, $u$ is called \emph{concave} if $-u$ is convex. When $u$ is $\C^2$, then it is convex if and only if $(u\circ\alpha)''\geq 0$, or equivalently the Hessian of $u$ is positive semidefinite.
We may also say that $u$ is convex on a set $X\subset M$ provided that $u$ is convex on all geodesic segments of $M$ contained in $X$. For basic facts and background on convex sets and functions in general Riemannian manifolds see \cite{udriste1994}, for convex analysis in Cartan-Hadamard manifolds see \cites{bishop-oniel1969,shiga1984,ballmann-gromov-schroeder}, and more generally for Hadamard or CAT(0) spaces (i.e., metric spaces of nonpositive curvature), see \cites{lurie:notes, bacak2014,bridson-haefliger1999, ballman1995}. In particular it is well-known  that if $M$ is a Cartan-Hadamard manifold, then $d\colon M\times M\to\R$ is convex \cite[Prop. 2.2]{bridson-haefliger1999}, which in turn yields \cite[Cor. 2.5]{bridson-haefliger1999}:

\begin{lemma}[\cite{bridson-haefliger1999}]\label{lem:u-convex}
If $M$ is a Cartan-Hadamard manifold, and $X\subset M$ is a convex set, then $d_X$ is convex.
\end{lemma}

\noindent So it follows that geodesic spheres are convex in a Cartan-Hadamard manifold as they are level sets of the distance function from one point. Let $X\subset M$ be a bounded convex set with interior points.
If $M=\R^n$, then it is well-known that $\widehat{d}_{\partial X}$ is convex on $X$ and therefore on all of $M$ \cite[Lemma 10.1, Ch. 7]{delfour-zolesio2011}. More generally $\widehat{d}_{\partial X}$ will be convex on $X$ as long as the curvature of $M$ on $X$ is nonnegative \cite[Lem. 3.3]{sakai1996}. However, if the curvature of $M$ is strictly negative on $X$, then $\widehat{d}_{\partial X}$ may no longer be convex. This is the case, for instance, when $X$ is the region bounded in between a pair of non-intersecting geodesics in the hyperbolic plane. See \cite[p. 44]{gromov1991} for a general discussion of the relation between convexity of parallel hypersurfaces and the sign of curvature of $M$. Therefore we are led to make the following definition. We say that a hypersurface $\Gamma$ in $M$ is distance-convex or \emph{d-convex} provided that $\widehat{d}_\Gamma$ is convex on $\Omega$.

As far as we know, $d$-convex hypersurfaces  have not been specifically studied before; however, as we show below, they are generalizations of the well-known h-convex or horo-convex hypersurfaces \cite{izumiya2009,borisenko-miquel1999,heintze-ImHof1977,currier1989,FMPV2011}, which are defined as follows. A \emph{horosphere}, in a Cartan-Hadamard manifold, is the limit of a family of geodesic spheres whose radii goes to infinity, and a \emph{horoball} is the limit of the corresponding family of balls (thus horospheres are generalizations of hyperplanes in $\R^n$). The distance function of a horosphere, which is known as a \emph{Busemann function}, has been extensively studied. In particular it is well-known that it is convex and $\C^2$ \cite[Prop. 3.1 \& 3.2]{ballman1995}.
A hypersurface $\Gamma$ is called \emph{h-convex} provided that through each of its points there passes a horosphere which contains $\Gamma$, i.e., $\Gamma$ lies in the corresponding horoball. The convexity of the Busemann function yields:

\begin{lemma}
In a Cartan-Hadamard manifold, every $\C^{1,1}$ h-convex hypersurface $\Gamma$ is $d$-convex.
\end{lemma}
\begin{proof}
For points $q\in\Gamma$, let $S_{q}$ be the horosphere which passes through $q$ and contains $\Gamma$.
For points $p\in\Omega$, let $p^\circ$ be the footprint of $p$ on $\Gamma$, and let $S_{p^\circ}$ be the horosphere which passes through $p^\circ$ and contains $\Gamma$. Then
$$
\widehat{d}_\Gamma(p)=-d(p,\Gamma)=-d(p,p^\circ)=-d(p, S_{p^\circ})=\widehat{d}_{S_{p^\circ}}(p).
$$
On the other hand, since  $\Gamma$ lies inside $S_q$, for any point $p\in\Omega$, we have $d(p,\Gamma)\leq d(p,S_q)$. Thus
$$
\widehat{d}_{\Gamma}(p)=  -d(p,\Gamma)\geq     -d(p, S_q)=      \widehat{d}_{S_{q}}(p).
$$
So we have shown that
$$
\widehat{d}_\Gamma=\sup_{q\in\Gamma} \widehat{d}_{S_q},
$$
on $\Omega$. Since $\widehat{d}_{S_q}$ (being a Busemann function) is convex, it follows then that $\widehat{d}_\Gamma$ is convex on $\Omega$, which completes the proof. 
\end{proof}

The converse of the above lemma, however, is not true. For instance, for any geodesic segment in the hyperbolic plane, there exists $r>0$, such that the tubular hypersurface of radius $r$ about that segment  (which  is $d$-convex by Lemma \ref{lem:u-convex}) is not h-convex. So in summary we may record that, in a Cartan-Hadamard manifold,
$$
\big\{\text{$h$-convex hypersurfaces}\big\}\;\varsubsetneq\; \big\{\text{$d$-convex hypersurfaces}\big\}\;\varsubsetneq\;  \big\{\text{convex hypersurfaces}\big\}.
$$

The main aim of this section is to relate the total curvature of a convex hypersurface in an $n$-dimensional Cartan-Hadamard manifold to that of a $d$-convex hypersurface in an $(n+1)$-dimensional Cartan-Hadamard manifold. First note that if $M$ is a Cartan-Hadamard manifold, then $M\times \R$ is also a Cartan-Hadamard manifold, which contains $M$ as a totally geodesic hypersurface. For any  convex hypersurface $\Gamma\subset M$, bounding a convex domain $\Omega$, and $\e>0$, let $\tilde\Gamma_\e$ be the parallel hypersurface of $\Omega$ in $M\times\R$ of distance $\e$. Then $\tilde\Gamma_\e$ is a $d$-convex hypersurface in $M\times\R$ by Lemma \ref{lem:u-convex}. Note also that $\tilde\Gamma_\e$ is $\C^{1,1}$ by Lemma \ref{lem:GH}, so its total curvature is well-defined. 
In the next proposition we will apply some facts concerning evolution of the second fundamental form of parallel hypersurfaces and tubes, which is governed by Riccati's equation. A standard reference here is Gray \cite[Chap. 3]{gray2004}; see also \cite{karcher1989,ballmann2016}. We will use some computations from \cite{ge-tang2014} on Taylor expansion of the second fundamental form. For more extensive computations see \cite{mahmoudi-mazzeo-pacard2006}.

First let us fix our basic notation and sign conventions with regard to computation of curvature. Let $\Gamma$ be a $\C^{1,1}$ closed embedded hypersurface in $M$, bounding a designated domain $\Omega$ of $M$ as we discussed in Section \ref{sec:distance}.
Then the \emph{outward normal} $\nu$ of $\Gamma$ is a unit normal vector field along $\Gamma$ which points away from $\Omega$.
Let $p$ be a twice differentiable point of $\Gamma$, and $T_p\Gamma$ denote the tangent space of $\Gamma$ at $p$. Then the \emph{shape operator} $\mathcal{S}_p\colon T_p\Gamma\to T_p \Gamma$ of $\Gamma$ at $p$ with respect to $\nu$ is defined as 
\be\label{eq:shape}
\mathcal{S}_p(V):=\nabla_{V}\nu,
\ee
for $V\in T_p \Gamma$. Note that in a number of sources, including \cites{gray2004,ge-tang2014} which we refer to for some computations, the shape operator is defined as $-\nabla_V\nu$. Thus our principal curvatures will have opposite signs compared to those in \cites{gray2004,ge-tang2014}, which will effect the appearance of Riccati's equation below.
 The eigenvalues and eigenvectors of $\mathcal{S}_p$ then define  the \emph{principal curvatures} $\kappa_i(p)$ and \emph{principal directions} $E_i(p)$ of $\Gamma$ at $p$ respectively. So we have
 $$
\kappa_i(p)= \big\langle \mathcal{S}_p(E_i(p)), E_i(p)\big\rangle=  \big\langle\nabla_i\nu, E_i(p)\big\rangle.
 $$
  The \emph{Gauss-Kronecker} curvature of $\Gamma$ at $p$  is  given by 
 \be\label{eq:detSp}
 GK(p):=\det(\mathcal{S}_p)=\prod_{i=1}^{n-1} \kappa_i(p).
\ee
Finally,  \emph{total Gauss-Kronecker} curvature of $\Gamma$ is defined as
$$
\mathcal{G}(\Gamma):=\int_\Gamma GK d\sigma.
$$
We will always assume that the shape operator of $\Gamma$ is computed with respect to the outward normal. Thus when $\Gamma$ is convex,  its principal curvatures will be nonnegative. The main result of this section is as follows. For convenience we assume that $\Gamma$ is $\C^2$, which will be sufficient for our purposes; however, the proof can be extended to the $\C^{1,1}$ case with the aid of Lemma \ref{lem:riccati} which will be established later.

\begin{proposition}\label{prop:d-convex}
Let $\Gamma$ be a $\C^{2}$ convex hypersurface in a Cartan-Hadamard manifold $M^n$, bounding a convex domain $\Omega$, and $\tilde\Gamma_\e$ be the  parallel hypersurface of $\Omega$ at distance $\e$ in $M\times\R$. Then,
as $\e\to 0$,
$$
\frac{\mathcal{G}(\tilde\Gamma_\e)}{\textup{vol}(\S^n)}\;\;\longrightarrow\; \;\frac{\mathcal{G}(\Gamma)}{ \textup{vol}(\S^{n-1})}.
$$
In particular, if $\mathcal{G}(\tilde\Gamma_\e)\geq \textup{vol}(\S^n)$, then $\mathcal{G}(\Gamma)\geq \textup{vol}(\S^{n-1})$.
\end{proposition}

Recall that in $\R^n$ the total curvature of any convex hypersurface is equal to the volume of its Gauss map. So Proposition \ref{prop:d-convex} holds immediately when $M=\R^n$. The proof in the general case follows from tube formulas and properties of the gamma and beta functions as we describe below (see also Note \ref{note:alpha_n} which eliminates the use of special functions). Note that
$$
\textup{vol}(\S^{n-1})=n\omega_n, \quad\quad\text{where}\quad\quad \omega_n:=\textup{vol}(\mathbf{B}^n)=\frac{\pi^{n/2}}{(n/2)!}.
$$

\begin{proof}[Proof of Proposition \ref{prop:d-convex}]
For every point $q\in\tilde\Gamma_\e$ let $p$ be its (unique) footprint on $\cl(\Omega)=\Omega\cup\Gamma$. If $p\in\Omega$, then there exists an open neighborhood $U$ of $p$ in $\tilde\Gamma_\e$ which lies on $M\times\{\e\}$ or $M\times\{-\e\}$. 
So $GK^\e(q)=0$, since each hypersurface $M\times\{t\}\subset M\times \R$ is totally geodesic. Thus the only contribution to $\mathcal{G}(\tilde\Gamma_\e)$ comes from points $q\in \tilde\Gamma_\e$ whose footprint $p\in\Gamma$. This portion of $\tilde\Gamma_\e$ is the outer half of the tube of radius $\e$ around $\Gamma$, which we denote by $\textup{tube}_\e^+(\Gamma)$, see Figure \ref{fig:tube},
\begin{figure}[h]
\centering
\begin{overpic}[height=1.4in]{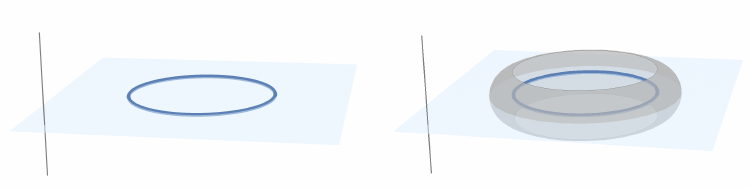}
\put(17,8.5){\Small $\Gamma$}
\put(26,11.5){\Small $\Omega$}
\put(42,7){\Small $M$}
\put(6,18){\Small $\R$}
\put(87,18){\Small $\textup{tube}_\e^+(\Gamma)$}
\put(92,6.5){\Small $M$}
\put(57,18){\Small $\R$}
\end{overpic}
\caption{}\label{fig:tube}
\end{figure}
and will describe precisely below. So we have
$$
\mathcal{G}(\tilde\Gamma_\e)=\mathcal{G}(\textup{tube}_\e^+(\Gamma)).
$$
Furthermore recall that $\omega_n=\pi^{n/2}/G(n/2+1)$, where $G$ is the gamma function. In particular, $G(1/2)=\sqrt\pi$, $G(x+1)=xG(x)$, and $G(n)=(n-1)!$, which yields
\be \label{eq:alpha_n}
\alpha_n:=\frac{\vol(\S^n)}{\vol(\S^{n-1})}=\frac{(n+1)\omega_{n+1}}{n\omega_n}=\frac{G(\frac12)G(\frac{n}{2})}{G(\frac12+\frac{n}2)}=B\left(\frac12,\frac{n}2\right)=\int_{-\pi/2}^{\pi/2}\cos^{n-1}(\theta)\,d\theta,
\ee
where $B$ is the beta function (see \cite[Sec. 1.1]{AAR1999} for the basic properties of gamma and beta functions).
Thus it suffices to  show that, as $\e\to 0$,
$$
\mathcal{G}\big(\textup{tube}_\e^+(\Gamma)\big)\;\to \alpha_n\,\mathcal{G}(\Gamma).
$$
To this end let $\nu$ denote the outward unit normal of $\Gamma$ with respect to $\Omega$ in $M$, $\nu^\perp$ be a unit normal vector orthogonal to $M$ in $M\times\R$, and  define $f^\e\colon \Gamma\times \R\to M\times\R$ by
$$
f^\e(p,\theta):=\exp_p\big(\e\nu_p(\theta)\big),\quad\quad \nu_p(\theta):=\cos(\theta)\nu(p)+\sin(\theta)\nu^\perp(p),
$$
where $\exp$ is the exponential map of $M\times\R$.
Then we set
$$
\textup{tube}_\e^+(\Gamma):=f^\e\big(\Gamma\times [-\pi/2,\pi/2]\big).
$$
Note that $\textup{tube}_\e^+(\Gamma)\subset d_\Gamma^{-1}(\e)$, where $d_\Gamma$ denotes the distance function of $\Gamma$ in $M\times\R$. Thus, since $M$ is $\C^2$, $d_\Gamma$ is $\C^2$ \cite[Thm. 1]{foote1984} which yields that  $\textup{tube}_\e^+(\Gamma)$ is $\C^2$. So the shape operator of $\textup{tube}_\e^+(\Gamma)$ is well-defined.
 By \cite[Cor. 2.2]{ge-tang2014}, this shape operator, at the point $f^\e(p,\theta)$, is given by
\be\label{eq:Spr}
\mathcal{S}^\e_{p,\theta}=
\left(
\begin{array}{cc}
\mathcal{S}_{p, \theta} +\mathcal{O}(\e)\;& \mathcal{O}(\e)\\
&\\
\mathcal{O}(\e) & 1/\e +\mathcal{O}(\e)
\end{array}
\right),
\ee
where $\mathcal{O}(\e)\to 0$ as $\e\to 0$, and $\mathcal{S}_{p, \theta}$ denotes the shape operator of $\Gamma$ at $p$ in the direction $\nu_p(\theta)$ (note that the shape operators in this work, as defined by \eqref{eq:shape}, have the opposite sign compared to those in \cite{ge-tang2014}). The eigenvalues of $\mathcal{S}_{p, \theta}$ are $\kappa_i(p)\cos(\theta)$ where $\kappa_i(p)$ are the principal curvatures of $\Gamma$ at $p$. Thus it follows that the Gauss-Kronecker curvature of $\textup{tube}_\e^+(\Gamma)$ at the point $f^\e(p,\theta)$ is given by
\begin{equation}\label{eq:GKrp}
GK^\e(p,\theta)=\det\big(\mathcal{S}^\e_{p,\theta}\big)=\frac{1}{\e}\det\big(\mathcal{S}_{p, \theta}\big)+\mathcal{O}(1)=\frac{1}{\e}GK(p)\cos^{n-1}(\theta)+\mathcal{O}(1),
\end{equation}
where $\mathcal{O}(1)$ converges to a constant as $\e\to 0$.
Furthermore, we claim that 
\be\label{eq:jac}
\text{Jac}(f^\e)_{(p,\theta)}=\e+\mathcal{O}(\e^2).
\ee
 Then it follows that, as $\e\to 0$,
$$
\mathcal{G}\big(\textup{tube}_\e^+(\Gamma)\big)
=\int_{\textup{tube}_\e^+(\Gamma)} GK^\e d\mu_\e
\to\int_{-\pi/2}^{\pi/2} \int_{p\in\Gamma}GK(p)\cos^{n-1}(\theta)\,d\mu d\theta
=\alpha_n\,\mathcal{G}(\Gamma),
$$
as desired. So it remains to establish \eqref{eq:jac}. To this end we will apply the fact that, due to Riccati's equation
\cite[Thm. 3.11 \& Lem. 3.12]{gray2004}, 
$$
\text{Jac}(f^\e)_{(p,\theta)}=\e\,\Theta(\e),
$$
where $\Theta$ is given by
\be\label{eq:gray}
\frac{\Theta'(\e)}{\Theta(\e)}=-\frac{1}{\e}+\text{trace}(\mathcal{S}^\e_{p,\theta}),\quad\quad \Theta(0)=1,
\ee
(again note that our shape operator has the opposite sign to that in \cite{gray2004}).  Next observe that by \eqref{eq:Spr}
$$
\text{trace}(\mathcal{S}^\e_{p,\theta})=\text{trace}(\mathcal{S}_{p,\theta})+\frac{1}{\e}+\mathcal{O}(\e). 
$$
So we may rewrite \eqref{eq:gray} as
$$
\frac{\Theta'(\e)}{\Theta(\e)}=\text{trace}(\mathcal{S}_{p,\theta})+\mathcal{O}(\e)=\mathcal{O}(1).
$$
 Hence,  we obtain 
$$
\Theta(\e)=\Theta(0)e^{\int_0^\e \mathcal{O}(1)dt} =e^{\mathcal{O}(\e)}=1+\mathcal{O}(\e),
$$
which in turn yields
$$
\text{Jac}(f^\e)_{(p,\theta)}=\e\big(1+\mathcal{O}(\e)\big)=\e+\mathcal{O}(\e^2),
$$
as desired.
\end{proof}

\begin{corollary}
If the total curvature inequality \eqref{eq:GK} holds for d-convex hyeprsurfaces, then it holds for all convex hypersurfaces.
\end{corollary}

\begin{corollary}
If the total curvature inequality \eqref{eq:GK} holds in dimension $n$, then it holds in all dimensions less than $n$.
\end{corollary}

\begin{note} \label{note:alpha_n}
Let $M=\R^n,\,\Gamma=\S^{n-1}$, and $\tilde\Gamma_\e$  be as in the statement of Proposition \ref{prop:d-convex}. Then the map $f^\e(p,\theta)$ in the  proof of Proposition \ref{prop:d-convex} 
simplifies to $p+\e\nu_p(\theta)$, and we quickly obtain
$$
\mathcal{G}(\tilde\Gamma_\e)= \left(\int_{-\pi/2}^{\pi/2}\cos^{n-1}(\theta)\,d\theta \right) \mathcal{G}(\Gamma).
$$
Since  $\mathcal{G}(\tilde \Gamma_{\e})=\vol(\S^{n})$ and 
$\mathcal{G}(\Gamma)=\vol(\S^{n-1})$, it follows that
$$
\frac{\vol(\S^{n})}{\vol(\S^{n-1})}=\int_{-\pi/2}^{\pi/2}\cos^{n-1}(\theta)\,d\theta,
$$
which establishes \eqref{eq:alpha_n} without the need to use the gamma and beta functions.
\end{note}

\begin{note}
An anonymous referee has brought to our attention that continuity properties for total curvature, as considered in Proposition \ref{prop:d-convex} above and some other questions discussed later in Section \ref{sec:CH}, may be treated via techniques in integral geometry; specifically, by using the fact that total curvature is given by a valuation, as stated for instance in \cite[Prop. 3.8]{bering-brocker}.
\end{note}

\section{The Comparison Formula}\label{sec:comparison}
In this section we establish an integral formula for comparing the total curvature of regular level sets of $\C^{1,1}$ functions on Riemannian manifolds. 
Here we assume that $u$ is a $\C^{1,1}$ function on a Riemannian manifold $M$, so that it is twice differentiable  almost everywhere, and
derive a basic identity for the cofactor operator, denoted by $\mathcal{T}^{\,u}$, associated to the Hessian of $u$. 
This operator is a special case of a more general device, the Newton operator, which appears in the well-known works or Reilly \cite{reilly1973,reilly1977}. To start,
let $\nabla$ be the \emph{covariant derivative} on $M$. The \emph{gradient} of $u$, $\nabla u$, is then given by
$$
\langle \nabla u(p), V\rangle :=\nabla_V u,
$$
for tangent vectors $V\in T_p M$.
Next (at a twice differentiable point $p$) we define the \emph{Hessian operator} $\nabla^2 u$ of $u$ as the self-adjoint linear map on $T_p M$ given by
$$
\nabla^2 u(V):=\nabla_V(\nabla u).
$$
The \emph{Hessian} of $u$ in turn will be the corresponding symmetric bilinear form on $T_p M$, 
$$
\textup{Hess}_u(V,W):=\left\langle \nabla^2 u(V),W \right\rangle=\left\langle \nabla_V(\nabla u),W \right\rangle.
$$
 Let $E_i$ denote a smooth orthonormal frame field in a neighborhood $U$ of $p$, and set $\nabla_i:=\nabla_{E_i}$.   Then $\nabla u=u_i E_i$ on $U$, and  $\nabla^2u(V)=u_{ij} V^j E_i$ at $p$, where 
$$
u_i:= \nabla_i u=\langle \nabla u, E_i\rangle,\quad\quad\text{and}\quad\quad u_{ij}:=\textup{Hess}_u(E_i, E_j).
$$
In general $u_{ij} = \nabla_j u_i - \langle \nabla_j E_i, E_k \rangle u_k $. We may assume, however, that
$(\nabla_j E_i)_p:=\nabla_{E_j(p)} E_i=0$, i.e., $E_i$ is a local \emph{geodesic frame} based at $p$.
Then
\begin{equation}\label{eq:nablaiej}
u_{ij}(p)= (\nabla_j u_i)_p.
\end{equation}
The  \emph{cofactor} of a square matrix $(a_{ij})$ is the matrix $(\ol a_{ij})$ where $\ol a_{ij}$ is  the $(i,j)$-\emph{signed minor} of $(a_{ij})$, i.e., $(-1)^{i+j}$ times the determinant of the matrix obtained by removing the $i^{th}$ row and $j^{th}$ column of $(a_{ij})$. We define the self-adjoint operator $\mathcal{T}^{\,u}\colon T_p M\to T_pM$ by setting
$$
(\mathcal{T}^{\,u}_{ij}):=\textup{cofactor}(u_{ij})=(\ol u_{ij}). 
$$
Note that, when $\nabla^2 u$ is nondegenerate, $(\nabla^2 u)^{-1}(V)=u^{ij} V^j E_i$, where 
$(u^{ij}):=(u_{ij})^{-1}$.
In that case,
$$
\mathcal{T}^{\,u}(V)= \det(\nabla^2 u)
(\nabla^2 u)^{-1}(V)=\mathcal{T}^{\,u}_{ij} V^j E_i,
$$  
and $(\mathcal{T}^{\,u}_{ij})=\det(\nabla^2u) (u^{ij})$.    We are interested in $\mathcal{T}^{\,u}$ since
it can be used to compute the curvature of the level sets of $u$, as discussed below.

We say that $\Gamma:=\{u=u(p)\}$ is a \emph{regular level set} of $u$ near $p$, if $u$ is $\C^1$ on a neighborhood of $p$ and $\nabla u(p)\neq 0$. Then $\nabla u/|\nabla u|$ generates a normal vector field on $\Gamma$ near $p$. If  we let $E_\ell$ be the principal directions of $\Gamma$ at $p$, then the corresponding principal curvatures of $\Gamma$ with respect to $\nabla u/|\nabla u|$ are given by

\begin{equation}\label{eq:kalpha}
\kappa_{\ell}=\left\langle \nabla_\ell \left(\frac{\nabla u}{|\nabla u|}\right), E_\ell\right\rangle= \frac{\langle \nabla_\ell (\nabla u), E_\ell\rangle}{|\nabla u|} =\frac{\textup{Hess}_u(E_\ell,E_\ell)}{|\nabla u|}= \frac{u_{\ell\ell}}{|\nabla u|}.
\end{equation}
Using this formula, we can show:

\begin{lemma}\label{lem:GK}
Let $\Gamma:=\{u=u(p)\}$ be a  level set of $u$ which is regular near $p$, and suppose that $\Gamma$ is twice differentiable at $p$. Then the Gauss Kronecker curvature of $\Gamma$ at $p$ with respect to $\nabla u/|\nabla u|$ is given by
$$
GK=\frac{\left\langle \mathcal{T}^{\,u}(\nabla u),\nabla u\right\rangle}{|\nabla u|^{n+1}}.
$$
\end{lemma}
\begin{proof}
Let  $E_i$ be an orthonormal frame for $T_pM$ such that $E_\ell$, $\ell=1,\dots,n-1$ are principal directions of  $\Gamma$ at $p$. 
Then the $(n-1)\times (n-1)$ leading principal submatrix of $(u_{ij})$ will be diagonal.
Thus, 
$$
\mathcal{T}^{\,u}_{nn}=\ol{u}_{nn}= \prod_{\ell=1}^{n-1}u_{\ell\ell}.
$$
Furthermore, since $E_n$ is orthogonal to $\Gamma$, and $\Gamma$ is a level set of $u$, $\nabla u$ is parallel to $\pm E_n$. So
$u_n=\langle\nabla u, E_n\rangle=\pm|\nabla u|$. Now, using \eqref{eq:kalpha}, we have
$$
\frac{\left\langle \mathcal{T}^{\,u}(\nabla u),\nabla u\right\rangle}{|\nabla u|^{n+1}}
=\frac{\mathcal{T}^{\,u}_{ij}u_j u_i}{|\nabla u|^{n+1}} 
=\frac{\mathcal{T}^{\,u}_{nn}u_n u_n}{|\nabla u|^{n+1}} 
=\frac{\mathcal{T}^{\,u}_{nn}}{|\nabla u|^{n-1}}
=\prod_{\ell=1}^{n-1}\frac{u_{\ell\ell}}{|\nabla u|}=GK.
$$ 
\end{proof}

Let $V$ be a vector field on $U$. Since $(\nabla_jE_i)_p=0$, the divergence of the vector field $\mathcal{T}^{\,u}(V)$ at $p$ is given by
\begin{equation}\label{eq:divAV}
\mbox{div}_p\big(\mathcal{T}^{\,u}(V)\big)=\big(\nabla_i(\mathcal{T}^{\,u}_{ij}V^j)\big)_p.
\end{equation} 
The \emph{divergence} $\textup{div}(\mathcal{T}^{\,u})$ of $\mathcal{T}^{\,u}$ is defined as follows. If $\mathcal{T}^{\,u}$ is viewed as a bilinear form or $(0,2)$ tensor, then $\textup{div}(\mathcal{T}^{\,u})$ generates a one-form or $(0,1)$ tensor given by $\langle \textup{div}(\mathcal{T}^{\,u}), \,\cdot\,\rangle$, where
\begin{equation}\label{eq:divA}
\textup{div}_p(\mathcal{T}^{\,u}):=(\nabla_i \mathcal{T}^{\,u}_{ij})_p\,E_j(p).
\end{equation}
In other words, with respect to our frame $E_i$,  $\textup{div}_p(\mathcal{T}^{\,u})$ is a vector whose $i^{th}$ coordinate is the divergence of the $i^{th}$ column of $\mathcal{T}^{\,u}$ at $p$.

\begin{lemma}\label{lem:divTu}
If  $u$ is three times differentiable at $p$, $\nabla u(p)\neq 0$, and $\nabla^2u(p)$ is nondegenerate, then
\be\label{eq:DI}
\textup{div}\left(\mathcal{T}^{\,u}\left(\frac{\nabla u}{|\nabla u|^n}\right)\right)
=\left\langle\textup{div}(\mathcal{T}^{\,u}),\frac{\nabla u}{|\nabla u|^n}\right\rangle.
\ee
\end{lemma}
\begin{proof} 
By \eqref{eq:divAV} and \eqref{eq:divA}, it suffices to check that, at the point $p$,
$$
\nabla_i\Big(\mathcal{T}^{\,u}_{ij}\frac{u_j}{|\nabla u|^n}\Big)=(\nabla_i \mathcal{T}^{\,u}_{ij})\frac{u_j}{|\nabla u|^n}.
$$
This follows  from \eqref{eq:nablaiej} via Leibniz rule, since
\[\mathcal{T}^{\,u}_{ij}\nabla_i\left(\frac{u_j}{|\nabla u|^n}\right)=\mathcal{T}^{\,u}_{ij}\left(\frac{u_{ji}}{|\nabla
u|^n}-n\frac{u_ju_ku_{ki}}{|\nabla u|^{n+2}}\right)=n\frac{\det(u_{ij})}{|\nabla
u|^n}-n\frac{u_ju_k\delta_{kj}\det(u_{ij})}{|\nabla u|^{n+2}}=0.
\]
\end{proof}

Next we apply the divergence identity \eqref{eq:DI} developed above  to obtain the comparison formula via Stokes' theorem. Let $\Gamma$ be a closed embedded $\C^{1,1}$ hypersurface in a Riemannian manifold $M$ bounding a domain $\Omega$. Recall that the outward normal  of $\Gamma$ is the unit normal vector field $\nu$ along $\Gamma$ which points away from $\Omega$, and if $p$ is a twice differentiable point of $\Gamma$, we will assume that  the Gauss-Kronecker curvature $GK(p)$ of $\Gamma$ is computed with respect to $\nu$ according to  \eqref{eq:detSp}. 
We say that $p$ is a \emph{regular point} of a function $u$ on $M$ provided that $u$ is $\C^1$ on an open neighborhood of $p$ and $\nabla u(p)\neq 0$. Furthermore, $x$ is a \emph{regular value} of $u$ provided that every $p\in u^{-1}(x)$ is a regular point of $u$. Then $u^{-1}(x)$ will be called a \emph{regular level set} of $u$.
In this section we assume that $\Gamma$ is a regular level set of $u$, and $\gamma$ is another regular level set  bounding a domain $D\subset\Omega$. We assume that $u$ is $\C^{2,1}$ on $\cl(\Omega)\setminus D$ and $\nabla u$ points outward along $\Gamma$ and $\gamma$ with respect to their corresponding domains. Furthermore we assume that $|\nabla u|\neq 0$ and $\nabla^2 u$ is nondegenerate at almost every point $p$ in $\cl(\Omega)\setminus D$.  Below we will assume that local calculations always take place at such a point $p$ with respect to a geodesic frame based at $p$, as defined above, and often omit the explicit reference to $p$. Throughout the paper $d\mu$ denotes the $n$-dimensional Riemannian volume measure on $M$, and $d\sigma$ is the $(n-1)$-dimensional volume or hypersurface area measure.

\begin{lemma}\label{lem2.2} 
$$
\mathcal{G}(\Gamma)-\mathcal{G}(\gamma)=\int_{\Omega\setminus D}
\left\langle\textup{div}(\mathcal{T}^{\,u}),\frac{\nabla u}{|\nabla u|^n}\right\rangle d\mu.
$$
\end{lemma}
\begin{proof} 
By Lemma \ref{lem:divTu} and the divergence theorem, 
\begin{eqnarray*}
\int_{\Omega\setminus D}
\left\langle\textup{div}(\mathcal{T}^{\,u}),\frac{\nabla u}{|\nabla u|^n}\right\rangle d\mu
&=&
\int_{\Omega\setminus D}
\textup{div}\left(\mathcal{T}^{\,u}\left(\frac{\nabla u}{|\nabla u|^n}\right)\right)d\mu\\
&=&\int_{\Gamma\cup\gamma}\left\langle \mathcal{T}^{\,u}\left(\frac{\nabla u}{|\nabla u|^n}\right),\nu\right\rangle d\sigma,
\end{eqnarray*}
 where $\nu$ is the outward normal to $\partial(\Omega\setminus D)=\Gamma\cup\gamma$. 
 Now Lemma \ref{lem:GK} completes the proof since by assumption $\nu=\nabla u/|\nabla u|$ on $\Gamma$ and $\nu=-\nabla u/|\nabla u|$ on $\gamma$.
\end{proof}

In the next computation we will need the formula
\begin{equation}\label{eq:nablaidet}
\nabla_i \det(\nabla^2u)=\mathcal{T}^{\,u}_{r\ell}u_{r\ell i}=\det(\nabla^2 u)u^{r\ell}u_{r\ell i},
\end{equation}
where $u_{rki}:=\nabla_iu_{rk}$.
Further note that by the definition of the Riemann tensor $R$ in local coordinates:
\begin{equation}\label{eq:urik}
u_{rik}-u_{rki}=\nabla_k\nabla_i u_r-\nabla_i\nabla_k u_r=R_{kir\ell}u_\ell,
\end{equation}
where we have used the fact that 
$R^\ell_{kir}=R_{kirm}g^{m\ell}=R_{kir\ell}$, since $g^{m\ell}:=\langle E_m, E_\ell\rangle=\delta_{m\ell}$.
Note that in formulas below we use the \emph{Einstein summation convention}, i.e., we assume that any term with repeated indices is summed over that index with values ranging from $1$ to $n$, unless indicated otherwise. The next observation relates the divergence of the Hessian cofactor to a trace or contraction of the Riemann tensor:
\begin{lemma}\label{lem2.3}
For any orthonormal frame $E_i$ at a point $p\in\Omega$,
$$
\big\langle \textup{div}(\mathcal{T}^{\,u}),\nabla u\big\rangle=\frac{R\big(\mathcal{T}^{\,u}(\nabla u),\mathcal{T}^{\,u}(E_i), E_i, \nabla u\big)}{\det(\nabla^2 u)}
=\frac{R\big(\mathcal{T}^{\,u}(\nabla u),E_i, \mathcal{T}^{\,u}(E_i), \nabla u\big)}{\det(\nabla^2 u)}.
$$
\end{lemma}
\begin{proof}
Differentiating both sides of $u^{ir}u_{rk}u^{kj}=u^{ij}$,  we obtain
$
\nabla_i u^{ij}=-u^{ir}u^{kj}u_{rki} \label{eq2.40}.
$
This together with \eqref{eq:nablaidet} and \eqref{eq:urik} yields that
\begin{eqnarray*}
\nabla_i \mathcal{T}^{\,u}_{ij}&=&\nabla_i(u^{ij}\det(\nabla^2u))\\
&=&-u^{ir}u^{kj}u_{rki}\det(\nabla^2 u)+u^{ij}\det(\nabla^2 u)u^{r\ell}u_{r\ell i}\\
&=&\det(\nabla^2 u)u^{kj}u^{ir}R_{kir\ell}u_\ell,
\end{eqnarray*}
where passing from the second line to the third proceeds via reindexing  $i\goto k$, $\ell\goto i$,  in the second term of the second line. Thus by \eqref{eq:divA}
$$
\langle \textup{div}(\mathcal{T}^{\,u}),\nabla u\rangle=\nabla_i \mathcal{T}^{\,u}_{ij}u_j=\det(\nabla^2 u)u^{kj}u^{ir}R_{kir\ell}u_\ell u_j.
$$
It remains then to work on the right hand side of the last expression. To this end recall that $u^{ij} E_j=\mathcal{T}^{\,u}_{ij}E_j/\det(\nabla^2 u)=\mathcal{T}^{\,u}(E_i)/\det(\nabla^2 u)$. Thus
\begin{eqnarray}\label{eq:Rkirl}
\det(\nabla^2 u)u^{kj} u^{ir} R_{kir\ell} u_\ell u_j
&=&\det(\nabla^2 u)R(u^{kj}E_ku_j,E_i,  u^{ir}E_r, u_\ell E_\ell )  \\ \notag
 &=& \frac{R(\mathcal{T}^{\,u}(\nabla u),E_i, \mathcal{T}^{\,u}(E_i), \nabla u)}{\det(\nabla^2 u)}.  \notag
\end{eqnarray}
Note that we may move $u_{ir}$, on the right hand side of the first inequality in the last expression, next to $E_i$, which will have the effect of moving $\mathcal{T}^{\,u}$ over to the second slot of $R$ in the last line of the expression. 
\end{proof}

Combining Lemmas \ref{lem2.2} and \ref{lem2.3} we obtain the basic form of the comparison formula:

\begin{corollary}\label{cor2.4}
Let $E_i$ be any choice of an orthonormal frame at each point $p\in\Omega\setminus D$. Then
$$
\mathcal{G}(\Gamma)-\mathcal{G}(\gamma)=\int_{\Omega\setminus D}
\frac{R\big(\mathcal{T}^{\,u}(\nabla u),\mathcal{T}^{\,u}(E_i), E_i, \nabla u\big)}{|\nabla u|^n\det(\nabla^2 u)} \,d\mu.
$$
\end{corollary}

Next we will express the integral  in Corollary \ref{cor2.4} with respect to a suitable local frame. To this end we need to gather some basic facts from matrix algebra: 

\begin{lemma}\label{lem2.4} 
Let $A$ be an $n\times n$ symmetric matrix, with diagonal 
$(n-1)\times (n-1)$ leading principal submatrix, given by
\begin{equation*}
\left(\;
\begin{array}{cccc}
b_1 &{} &{\makebox(-5,0){\textup{\huge0}}}  &a_1\\
{} &{\ddots} & {} & \vdots \\
{\makebox(0,15){\textup{\huge0}}}&{}  & b_{n-1}& a_{n-1}\\
a_1   & \cdots & a_{n-1} & a
\end{array}
\right),
\end{equation*}
and let $\ol A=(\ol{a}_{ij})$ denote the cofactor matrix of $A$. Then
\begin{enumerate}[(i)]
\item  $\ol{a}_{in}= -a_i \Pi_{\ell\neq i} b_\ell$ for $i<n$.
\item $\ol{a}_{ij}=a_i a_j \Pi_{\ell\neq i,j}b_\ell$ for $i,j<n~,~i\neq j$.
\item $\ol{a}_{ii}=a\Pi_{l\neq i}b_\ell-\sum_{k\neq  i}a_k^2~\Pi_{\ell\neq k,i}b_\ell$ for
$i<n$.
\item $\det(A)= a\Pi_\ell b_\ell -\sum a_k^2~\Pi_{\ell\neq k}b_\ell$
\item For fixed $b_1,\ldots, b_{n-1}, |a| $ tending to infinity, and $|a_i|
<C$ (independent of $a$), the eigenvalues of $A$ satisfy
$\lambda_{\alpha}=b_{\alpha}+o(1)$ for $\alpha<n$ and 
$\lambda_n=a+\mathcal{O}(1)$,
where the $o(1)$ and $\mathcal{O}(1)$ are uniform depending only on $~b_1, \ldots , b_{n-1}$
and $C$. In particular,
$$
\det(A)=a \prod_{i} b_i +\mathcal{O}(1).
$$
\end{enumerate}
\end{lemma}
\begin{proof} 
Parts (i), (ii), and (iii) follow easily by induction, and part (iv) follows
from part (i) by the cofactor expansion of $\det(A)$ using the last column. Finally, part (v) is provided by \cite[Lem. 1.2]{CNS1985}.
\end{proof}

Let $p$ be a regular point of a function $u$ on $M$. 
We say that $E_1,\dots, E_n\in T_p M$ is a \emph{principal frame}  of $u$ at  $p$ provided that 
$$
E_n=-\frac{\nabla u(p)}{|\nabla u(p)|},
$$ 
and $E_1,\dots, E_{n-1}$ are principal directions of the level set $\{u=u(p)\}$ at $p$ with respect to $-E_n$. Then the corresponding principal curvatures  and the Gauss-Kronecker curvature of $\{u=u(p)\}$ will be denoted by $\kappa_i(p)$ and $GK(p)$ respectively.
By a principal frame for $u$ over some domain we mean a choice of principal frame at each point of the domain.

\begin{theorem}[Comparison Formula, First Version]\label{thm:comparison}
Let $u$ be a  function on a Riemannian manifold $M$, and $\Gamma$, $\gamma$ be a pair of its regular level sets bounding domains $\Omega$, and $D$ respectively, with $ D\subset\Omega$. Suppose that $\nabla u$ points outward along $\Gamma$ and $\gamma$ with respect to their corresponding domains. Further suppose that 
$u$ is $\C^{2,1}$ on $\cl(\Omega)\setminus D$, and almost everywhere on $\cl(\Omega)\setminus D$,
$\nabla u\neq 0$, and $\nabla^2 e^u$ is nondegenerate. Then, 
\begin{eqnarray*}\label{eq2.90}
\mathcal{G}(\Gamma)-\mathcal{G}(\gamma)
=
-\int_{\Omega\setminus D}
R_{rnrn}\frac{GK}{\kappa_r}\,d\mu
+
\int_{\Omega\setminus D}
R_{rkrn}\frac{GK}{\kappa_r\kappa_k}\frac{u_{nk}}{|\nabla u|} d\mu,
\end{eqnarray*}
where all quantities are computed with respect to a principal frame of $u$,
and $k\leq n-1$.
\end{theorem}
\begin{proof}
Let $w:=\phi(u):=(e^{h u}-1)/h$ for $h>0$. Then $\gamma$, $\Gamma$ will be level sets of $w$ and $\nabla^2w$ will be nondegenerate almost everywhere. So we may apply Corollary \ref{cor2.4} to  $w$ to obtain
\begin{equation}
\mathcal{G}(\Gamma)-\mathcal{G}(\gamma)
=\int_{\Omega\setminus D}
\frac{R\big(\mathcal{T}^w(\nabla w),E_i, \mathcal{T}^w(E_i), \nabla w\big)}{\det(\nabla^2 w)|\nabla w|^n}d\mu.
\label{eq2.100}
\end{equation}
Let $p$ be a point of the level set $\{w=\phi(t)\}$, and   $E_\alpha$, $\alpha=1,\dots, n-1$ be principal directions of $\{w=\phi(t)\}$ at $p$. Since $w$ is constant on $\{w=\phi(t)\}$, $w_i(p)=0$  for $i<n$ and $|w_n|=|\nabla w|$. Consequently,  the integrand in the right hand side of (\ref{eq2.100}) at $p$ is given by
\begin{equation}\label{eq2.110}
\frac{\det(\nabla^2 w)w^{kj} w^{ir} R_{kir\ell} w_\ell w_j}{|\nabla w|^n}
= \frac{\det(\nabla^2 w)w^{kn} w^{ir} R_{kirn}}{|\nabla w|^{n-2}}.
\end{equation}
Next note that $(w_{ij})=\phi'(u)(a_{ij})$ where, $a_{ij}={u_{ij}+h u_i u_j}$. Again we have $u_i(p)=0$ for $i<n$. 
Also recall that, by \eqref{eq:kalpha},
$
u_{kk}=|\nabla u|\kappa_k.
$ 
Furthermore note that $\nabla u=-u_n$.
 Thus  it follows that
\begin{equation*}
(a_{ij})=\left(\begin{array}{cccc}
|\nabla u|\kappa_1&{}  & {\makebox(0,0){\text{\huge0}}} &u_{1n}\\
 {} &{} & {} & {} \\
{} &{\ddots} & {} & \vdots \\
{} &{} & {} & {} \\
{\makebox(0,10){\text{\huge0}}} & {} &|\nabla u|\kappa_{n-1}& u_{(n-1)n}\\
{} &{} & {} & {} \\
u_{n1}&  \cdots &u_{(n-1)n}& u_{nn}+h|\nabla u|^2
\end{array}\right).
\end{equation*}
Let $(\ol{a}_{ij})$  be the cofactor matrix of $(a_{ij})$. Since $(w_{ij})=\phi'(u)(a_{ij})$, it follows that the  cofactor matrix of $(w_{ij})$ is given by $\det(\nabla^2 w)w^{ij}=\phi'(u)^{n-1}\ol{a}_{ij}$. Then, the right hand side of
\eqref{eq2.110} becomes:
\begin{equation}\label{eq2.120}
\frac{\det(\nabla^2 w)w^{kn} w^{ir} R_{kirn}}{|\nabla w|^{n-2}}
=\frac{\phi'(u)^{2n-2}\ol{a}_{kn}\ol{a}_{ir} R_{kirn}}{\det(\nabla^2 w)|\nabla w|^{n-2}}
=\frac{\ol{a}_{kn}\ol{a}_{ir}R_{kirn}}{\det(a_{ij})|\nabla  u|^{n-2}},
\end{equation}
where in deriving the second equality we have used the facts that $|\nabla w|=\phi'(u)|\nabla u|$, and $\det(\nabla^2 w)=\phi'(u)^n\det(a_{ij})$.
 By Lemma \ref{lem2.4} (as $h\goto \infty$),
\begin{eqnarray*}
\ol{a}_{ij}=
\begin{cases}
-u_{in}\frac{GK}{\kappa_i}|\nabla u|^{n-2}, & \mbox{for}~i<n \,\;\text{and}\;\;j=n;\\

u_{in}u_{nj}\frac{GK}{\kappa_i \kappa_j}|\nabla u|^{n-3}, & \mbox{for}~
i\neq j\;\;\text{and}\;\; i,j<n;\\

\left(u_{nn}+h|\nabla u|^2\right)\frac{GK}{\kappa_i}|\nabla u|^{n-2}+\mathcal{O}(1), & \mbox{for}~i=j\;\; \text{and}\;\; i,j<n;\\

GK|\nabla u|^{n-1}, & \text{for}\,\, i=j=n.\\
\end{cases}
\end{eqnarray*}
Observe that $\ol{a}_{ij}$ for $i\neq j$ or $i=j=n$ 
are independent of $h$. On the other hand, again by Lemma \ref{lem2.4},
$$
\det(a_{ij})=\big(u_{nn}+h|\nabla u|^2\big)GK|\nabla u|^{n-1}+\mathcal{O}(1).
$$
 Therefore, 
the last term in \eqref{eq2.120} takes the form
$$
\frac{\ol{a}_{kn}\ol{a}_{rr}R_{krrn}}{\det(a_{ij})|\nabla  u|^{n-2}}+\mathcal{O}\left(\frac1{h}\right) \label{eq2.170}\\
\;=\; -R_{rnrn}\frac{GK}{\kappa_r}+R_{rkrn}\frac{GK}{\kappa_r  \kappa_k}
\frac{u_{nk}}{|\nabla u|}+\mathcal{O}\left(\frac1{h}\right). \label{eq2.180}\\ 
$$
where $k\leq n-1$.
So,  by the coarea formula, the right hand side of \eqref{eq2.100} becomes
\begin{gather*}
\int_{\phi(t_0)}^{\phi(t_1)} 
\int_{\{w=s\}}
\left(-R_{rnrn}\frac{GK}{\kappa_r}+R_{rkrn}\frac{GK}{\kappa_r \kappa_k}
\frac{u_{nk}}
{|\nabla u|}+\mathcal{O}\left(\frac1{h}\right)\right)\frac{d\sigma}{|\nabla w|}ds\\
=\int_{t_0}^{t_1} \int_{\{u=t\}}
\left(-R_{rnrn}\frac{GK}{\kappa_r}+R_{rkrn}\frac{GK}{\kappa_r \kappa_k}
\frac{u_{nk}}{|\nabla u|}\right)\frac{d\sigma}{|\nabla u|}dt
\end{gather*}
 after the change of variable $s=\phi(t)$ and letting $h\goto
\infty$. 
\end{proof}

Next we develop a more general version of Theorem \ref{thm:comparison}, via integration by parts and a smoothing procedure, which may be applied to $\C^{1,1}$ functions, to functions with singularities, or to a sequence of functions whose derivatives might blow up over some region. 
First we describe the smoothing procedure.  Let $\rho(x):=d(x,x_0)$, for some $x_0\in\Omega$, and set
$$
\ol u^{\e}(x):=u(x)+\frac{\e}2 \rho^2(x).
$$
If $u$ is convex, then $\ol u^{\e}$ will be strictly convex in the sense of Greene and Wu \cite{greene-wu1972},  and thus their method of smoothing by convolution will preserve convexity of $u$. This convolution is a generalization of the standard Euclidean version via the exponential map, and is defined as follows. Let $\phi: \R \to \R$ be a nonnegative $\C^{\infty}$ function supported in $[-1,1]$ which is constant in a neighborhood of the origin, and satisfies $\int_{\R^n}\phi(|x|)dx=1$. Then for any function $f\colon M\to\R$, we set
\be \label{eq4.10}  
f\circ_\lambda\phi(p):=\frac{1}{\lambda^{n}}\int_{v\in T_pM} \phi\left(\frac{|v|}{\lambda}\right)f\big(\text{exp}_p(v)\big)\,d\mu_p,
\ee
where $d\mu_p$ is the measure on $T_pM\simeq\R^n$ induced by the Riemannian measure $d\mu$ of $M$. We set
$$
\widehat u^{\,\e}_\lambda:=\ol u^{\,\epsilon}\circ_\lambda \phi.
$$
The following result is established in \cite[Thm. 2 \& Lem. 3(3)]{greene-wu1976}, with reference to earlier work in \cites{greene-wu1972, greene-wu1974}. In particular see \cite[p. 280]{greene-wu1974} for how differentiation under the integral sign in \eqref{eq4.10} may be carried out via parallel translation.

\begin{proposition} \label{prop4.1}(Greene-Wu \cite{greene-wu1976})
 For any continuous function $u\colon M\to\R$,  $\e>0$, and compact set $X\subset M$, there exists $\lambda>0$ such that $\widehat u_{\lambda}^{\,\e}$ is $\C^\infty$ on an open
  neighborhood $U$ of $X$, and  $\widehat u_{\lambda}^{\,\e}\to \ol u^{\,\e}$
  uniformly on $U$, as $\lambda\to 0$. Furthermore, if $u$ is $\C^k$ on an open neighborhood of $X$, then
$\widehat u_{\lambda}^{\,\e}\to \ol u^{\,\epsilon}$ on $U$  with respect to the $\C^k$ topology. Finally, if $u$ is convex, then $\widehat u_{\lambda}^{\,\e}$ will be strictly convex with positive definite Hessian everywhere.
\end{proposition}

Recall that, for any set $A\subset M$,
$
U_\theta(A)
$
denotes the tubular neighborhood of radius $\theta$ about $A$.
 A \emph{cutoff function} for $U_\theta(A)$ is a continuous  function $\eta\geq 0$ on $M$ which depends only on the distance $\widehat r(\,\cdot\,):=d_A(\,\cdot\,)$, is nondecreasing in terms of $\widehat r$,  and satisfies
\be\label{eq:eta-def}
\eta(x):=
\begin{cases}
0 \quad \text{if}\quad \widehat r(x)\leq \theta,\\
1 \quad \text{if}\quad \widehat r(x)\geq 2\theta.
\end{cases}
\ee
Since by Lemma \ref{lem:lipschitz} $\widehat r$ is Lipschitz, we may choose $\eta$ to be Lipschitz as well, and thus differentiable almost everywhere. At every differentiable point of $\eta$ we have
\begin{eqnarray*}
\left\langle \mathcal{T}^{\,u}\left(\frac{\nabla u}{|\nabla u|^n}\right),\nabla\eta\right\rangle
=\frac{\mathcal{T}^{\,u}_{ij}\eta_iu_j}{|\nabla u|^n}
=\frac{\mathcal{T}^{\,u}_{in}\eta_iu_n}{|\nabla u|^n}
=-\frac{\mathcal{T}^{\,u}_{in}\eta_i}{|\nabla u|^{n-1}}
=-\frac{\mathcal{T}^{\,u}_{kn}\eta_k}{|\nabla u|^{n-1}}-  \frac{\mathcal{T}^{\,u}_{nn}\eta_n}{|\nabla u|^{n-1}},
\end{eqnarray*}
where $k\leq n-1$.
Furthermore, by Lemma \ref{lem2.4},
$$
-\frac{\mathcal{T}^{\,u}_{kn}\eta_k}{|\nabla u|^{n-1}}
=\frac{u_{nk}\eta_k}{|\nabla u|} \prod_{\ell\neq k}\kappa_\ell=\frac{u_{nk}\eta_k}{|\nabla u|} \frac{GK}{\kappa_k},
\quad\quad\text{and}\quad\quad
-\frac{\mathcal{T}^{\,u}_{nn}\eta_n}{|\nabla u|^{n-1}}=-\eta_nGK.
$$
So we obtain
$$
\left\langle \mathcal{T}^{\,u}\left(\frac{\nabla u}{|\nabla u|^n}\right),\nabla\eta\right\rangle
= \frac{u_{nk}\eta_k}{|\nabla u|} \frac{GK}{\kappa_k}- \eta_n GK.
$$
Next recall that 
$
\int\textup{div}(\eta Y) d\mu
=\int(\left\langle Y,\nabla\eta\right\rangle+
\eta\,\textup{div}(Y))d\mu, 
$
for any vector field $Y$ on $M$. Thus
\begin{multline} \label{eq250}
\int\textup{div}\left(\eta\, \mathcal{T}^{\,u}\left(\frac{\nabla u}{|\nabla u|^n}\right)\right)d\mu=\\
\int \left(\frac{u_{nk}\eta_k}{|\nabla u|} \frac{GK}{\kappa_k}- \eta_n GK \right)d\mu+
\int\eta\,\textup{div}\left(\mathcal{T}^{\,u}\left(\frac{\nabla u}{|\nabla u|^n}\right)\right) d\mu.
\end{multline}
We set
$$
\mathcal{G}_\eta(\Gamma):=\int_\Gamma \eta\, GK\,d\sigma,\quad\quad\text{and}\quad\quad\mathcal{G}_\eta(\gamma):=\int_\gamma \eta\, GK\,d\sigma.
$$
The following result generalizes the comparison formula in Theorem  \ref{thm:comparison}. Note in particular that our new comparison formula may be applied to convex functions, where the principal curvatures of  level sets might vanish. 
So we will use the following conventions.
\be\label{eq:conventions}
\frac{GK}{\kappa_r}:=\prod_{i\neq r}\kappa_i,\quad\quad\text{and}\quad \quad \frac{GK}{\kappa_r\kappa_k}:=\prod_{i\neq r, k}\kappa_i,
\ee
Now the terms $GK/\kappa_r$ and $GK/(\kappa_r\kappa_k)$ below will always be well-defined.

\begin{theorem}[Comparison Formula, General Version]\label{thm:comparison2}
Let $u$, $\Gamma$, $\gamma$, $\Omega$, and $D$ be as in Theorem  \ref{thm:comparison}, except that $u$ is  $\C^{1,1}$  on $(\Omega\setminus  D)\setminus A$, for some (possibly empty) closed  set $A\subset \Omega\setminus  D$, and $u$ is either convex or else $\nabla^2e^u$ is nondegenerate almost everywhere on $(\Omega\setminus  D)\setminus A$. Then, for any $\theta>0$, and cutoff function $\eta$ for $U_\theta(A)$,
\begin{multline*}
\mathcal{G}_\eta(\Gamma)-\mathcal{G}_\eta(\gamma)=\\
\int_{\Omega\setminus  D}  \left(\eta_k \frac{GK}{\kappa_k}\frac{u_{nk}}{|\nabla u|}- \eta_n GK \right)d\mu+
\int_{\Omega\setminus  D}\eta\left(-R_{rnrn}\frac{GK}{\kappa_r}
+R_{rkrn}\frac{GK}{\kappa_r \kappa_k}\frac{u_{nk}}{|\nabla u|}\right) d\mu,
\end{multline*}
where all quantities are computed with respect to a principal frame of $u$,
and $k\leq n-1$.
\end{theorem}
\begin{proof}
Let $\widehat u_{\lambda}^{\,\e}$ be as in Proposition \ref{prop4.1} with $X$ in that theorem set to $\cl(\Omega)\setminus D$. Furthermore,
let $\Gamma_{\lambda}^{\e}$ and $\gamma_{\lambda}^{\e}$ be regular level sets of $\widehat u_{\lambda}^{\,\e}$ close to $\Gamma$ and $\gamma$ respectively.
Replace $u$ by $\widehat u_{\lambda}^{\,\e}$ in \eqref{eq250} and follow  
virtually the same argument used in Theorem \ref{thm:comparison}. Finally, letting $\lambda$ and then $\e$ go to $0$ completes the argument.
\end{proof}

\section{Applications of the Comparison Formula}\label{sec:applications}

Here we will record  some consequences of the comparison formula developed in Theorem \ref{thm:comparison2}.
Let
\be\label{eq:sigma_r-def}
\sigma_r(x_1,\dots, x_k):=\sum_{i_1<\dots <i_r} x_{i_1}\dots x_{i_r},
\ee
denote the \emph{elementary symmetric functions}. Furthermore, set $\kappa:=(\kappa_1,\dots, \kappa_{n-1})$, where $\kappa_i$ are principal curvatures of  level sets $\{u=u(p)\}$ at a regular point $p$ of $u\colon M\to\R$ which is twice differentiable. Then the $r^{th}$ \emph{generalized mean curvature} of $\{u=u(p)\}$ is given by
$$
\sigma_r(\kappa):=\sigma_r(\kappa_1,\dots, \kappa_{n-1}).
$$
In particular note that $\sigma_{n-1}(\kappa)=GK$, and $\sigma_1(\kappa)=(n-1)H$, where $H$ is the (normalized first) mean curvature of $\{u=u(p)\}$. The integrals of $\sigma_r(\kappa)$, which are called \emph{quermassintegrals}, are central in the theory of mixed volumes \cites{schneider2014, trudinger1994}.
 We also need to record the following basic fact. Recall that $n\omega_n=\textup{vol}(\S^{n-1})$.

\begin{lemma} \label{lem:GKnomegan} 
Let $B_r$ be a geodesic ball of radius $r$ in a Riemannian manifold $M^n$. Then  
$$
|\mathcal{G}(\partial B_r)-n\omega_n| \leq Cr^2,
$$
for some constant $C$ which is independent of $r$.
\end{lemma}
\begin{proof}
Since $\partial B_r$ is a geodesic sphere, a power series expansion of its second fundamental form in normal coordinates, see \cite[Thm. 3.1]{chen-vanhecke1981}, shows that  
$$
0\leq GK \leq \frac1{r^{n-1}}(1+Cr^2),
$$
where $GK$ denotes the Gauss-Kronecker curvature of $\partial B_r$. Furthermore another power series expansion \cite[Thm. 3.1]{gray1974} shows that  
$$
\left|\textup{vol}(\partial B_r)-n\omega_nr^{n-1}\right|\leq Cr^{n+1}.
$$
Using these inequalities we obtain
\begin{eqnarray*}
0\leq\mathcal{G}(\partial B_r)-n\omega_n &\leq& \frac1{r^{n-1}}(1+Cr^2)\textup{vol}(\partial B_r)-n\omega_n \\
&\leq& r^{1-n}(1+Cr^2)\cdot n\omega_n r^{n-1}(1+Cr^2)-n\omega_n\\
& \leq& n\omega_n(1+Cr^2-1) \\
&\leq& Cr^2 ,
\end{eqnarray*}
as desired.
\end{proof}

In the following corollaries we adopt the same notation as in Theorem \ref{thm:comparison2} and assume that $X=\emptyset$. The first corollary shows in particular that the total curvature inequality \eqref{eq:GK} holds in hyperbolic space:

\begin{corollary}[Nested hypersurfaces in space forms] \label{cor2.1} 
If $M$ has constant sectional curvature $K_0$, then
\be\label{eq:nested}
\mathcal{G}(\Gamma)-\mathcal{G}(\gamma)=-K_0\int_{\Omega\setminus  D} 
\sigma_{n-2}(\kappa)d\mu.
\ee
In particular, $\mathcal{G}(\Gamma)\geq \mathcal{G}(\gamma)$ if $\Gamma$, $\gamma$ are convex and $K_0\leq 0$. Furthermore, if $\Gamma$ is convex and $K_0\leq 0$, then 
\be\label{eq:nested2}
\mathcal{G}(\Gamma)\geq n\omega_n -K_0\int_{\Omega} 
\sigma_{n-2}(\kappa)d\mu\geq n\omega_n.
\ee
\end{corollary}
\begin{proof}
Since $M$ has constant sectional curvature $K_0$,  $R_{ijk\ell}=
K_0(\delta_{ik}\delta_{j\ell}-\delta_{i\ell}\delta_{jk})$. Thus \eqref{eq:nested} follows immediately from Theorem \ref{thm:comparison2}.
If $\Gamma$ and $\gamma$ are convex, then we may assume that $u$ is convex \cite[Lem. 1]{borbely2002}. In particular level set of $u$ are convex, and thus $\sigma_{n-2}(\kappa)\geq 0$. So $\mathcal{G}(\Gamma)\geq \mathcal{G}(\gamma)$. Finally, letting $\gamma$ be a sequence of geodesic balls with vanishing radius, we obtain \eqref{eq:nested2} by Lemma \ref{lem:GKnomegan}.
\end{proof}

The monotonicity property for total curvature of nested convex hypersurfaces in the hyperbolic space $\mathbf{H}^n$ had been observed earlier by Borbely \cite{borbely2002}.  In a general Cartan-Hadamard manifold, however, this property does not hold, as has been shown by Dekster \cite{dekster1981}. Thus the requirement that the curvature be constant in Corollary \ref{cor2.1} is essential.
Another important special case of Theorem \ref{thm:comparison2} occurs when $|\nabla u|$ is constant
on level sets of $u$, or $u_{kn}\equiv 0$ (for $k\leq n-1$), e.g., $u$ may be the distance function of $\Gamma$, in which case recall that we say $\gamma$ and $\Gamma$ are parallel. 

\begin{corollary}[Parallel hypersurfaces] \label{cor:parallel} 
Suppose that $u=\widehat{d}_\Gamma$. Then
\begin{equation} \label{eq2.200}
\mathcal{G}(\Gamma)-\mathcal{G}(\gamma)=-\int_{\Omega\setminus D}
R_{rnrn} \frac{GK}{\kappa_r}\,d\mu.
\end{equation} 
In particular, if  $\gamma$ is convex,  and $K_M\leq -a\leq 0$, then 
\be\label{eq:Gamma-geq-gamma}
\mathcal{G}(\Gamma)\geq \mathcal{G}(\gamma)+a\int_{\Omega\setminus D}
\sigma_{n-2}(\kappa)\,d\mu.
\ee
Finally if $\Gamma$ is a geodesic sphere and $K_M\leq 0$, then $\mathcal{G}(\Gamma)\geq n\omega_n$.
\end{corollary}
\begin{proof}
Inequality \eqref{eq2.200} follows immediately from Theorem \ref{thm:comparison2} (with $X=\emptyset$). 
If $\gamma$ is convex, then all of its outer parallel hypersurfaces, which are level sets of $u$ fibrating $\Omega\setminus D$, are convex as well by Lemma \ref{lem:u-convex}.
Thus $\sigma_{n-2}(\kappa)\geq 0$ which yields \eqref{eq:Gamma-geq-gamma}. Finally, if $\Gamma$ is a geodesic sphere, then $\gamma$ is a geodesic sphere as well, and letting radius of $\gamma$ shrink to $0$ we obtain $\mathcal{G}(\Gamma)\geq n\omega_n$ via Lemma \ref{lem:GKnomegan}. 
\end{proof}

The monotonicity of total curvature for parallel hypersurfaces in Cartan-Hadamard manifolds had been observed earlier by Schroeder-Strake \cite{schroeder-strake}; see also Cao-Escobar \cite[Prop. 2.3]{cao-escobar} for a version in polyhedral spaces, and Note \ref{note:monotone} for an alternative argument. Finally we sharpen the last statement of the above corollary with regard to geodesic spheres. To this end we need the following observation:

\begin{lemma}\label{lem:radial-curvature}
Let $U$ be an open set in a Cartan-Hadamard manifold $M$ which is star-shaped with respect to a point $p\in U$. Suppose that the curvature of $U$ is constant with respect to all planes which are tangent to the geodesics emanating from $p$. Then the curvature of $U$ is constant.
\end{lemma}
\begin{proof}
Let $K_0$ be the values of the constant curvature, and $\tilde M$ be a complete simply connected manifold of constant curvature $K_0$ and of the same dimension as $M$. Let $\tilde p\in \tilde M$, and $i\colon T_p M\to T_p\tilde M$ be an isometry. Define $f\colon U\to \tilde M$ by $f(q):=\exp_{\tilde p}\circ i\circ\exp_p^{-1}(q)$. A standard Jacobi field argument shows that $f$ is an isometry, e.g., see the proof of Cartan's theorem on determination of metric from curvature \cite[p. 157]{docarmo} (the key point here is that Jacobi's equation depends only on the sectional curvature with respect to the planes which are tangent to the geodesic).
\end{proof}

\begin{corollary}[Geodesic spheres]
Let $\Gamma=\partial B_\rho$, $u=\widehat{d}_\Gamma$, and suppose that $K_M\leq -a\leq 0$. Then
\begin{equation} \label{eq:corcor:parallel}
\mathcal{G}(\partial B_\rho)\geq 
n\omega_n+a\int_{B_\rho}
\sigma_{n-2}(\kappa)\,d\mu
\geq  \mathcal{G}(\partial B^{a}_\rho),
\end{equation}
where $B^{a}_\rho$ is a geodesic ball of radius
$\rho$ in the hyperbolic space $\mathbf{H}^n(-a)$.
Equality  holds in either of the above inequalities only if $B_\rho$ is isometric to $B_\rho^a$.
\end{corollary}
\begin{proof}
Let $B_{r}$ denote the geodesic ball of radius $r<\rho$ with the same center as $B_\rho$. By \eqref{eq:Gamma-geq-gamma},
$$
\mathcal{G}(\partial B_\rho)    \geq  \mathcal{G}(\partial B_{r})+a\int_{B_\rho\setminus B_r}
\sigma_{n-2}(\kappa)\,d\mu.
$$
Letting $r\to 0$, we obtain  the first inequality in \eqref{eq:corcor:parallel} via Lemma \ref{lem:GKnomegan}. Next assume that equality holds in \eqref{eq:corcor:parallel}. Then $R_{rnrn}=-a$. So $B_\rho$ has constant curvature $-a$ by Lemma \ref{lem:radial-curvature}.
Next to establish the second inequality in \eqref{eq:corcor:parallel} note that, by basic Riemannian comparison theory \cite[p. 184]{karcher1989}, principal curvatures of $\partial B_{r}$ are
 bounded below by $\sqrt{a}\coth(\sqrt{a}r)$. Hence, on $\partial B_r$,
 $$
 \sigma_{n-2}(\kappa)\geq (n-1)(\sqrt{a}\coth \sqrt{a} r )^{n-2}.
 $$
 Let $A(r,\theta)d\theta$ denote the volume (surface area) element of $\partial B_r$, and $H(r,\theta)$ be  its (normalized) mean curvature function in geodesic spherical coordinates (generated by the exponential map based at the center of $B_r$).
By \cite[(1.5.4)]{karcher1989}, 
$$
\frac{d}{dr}A(r,\theta) =(n-1) H(r,\theta)A(r,\theta) \geq (n-1)\sqrt{a} \coth(\sqrt{a}r)\, A(r,\theta),
$$ 
which after an integration yields
$$
A(r,\theta)\geq \left(\frac{\sinh(\sqrt{a}r)}{\sqrt{a}}\right)^{n-1}.
$$
Thus,
\begin{eqnarray*}
\mathcal{G}(\partial B_\rho)
&\geq&
n\omega_n+ a\int_{0}^\rho \int_{\S^{n-1}} \sigma_{n-2}(\kappa) A(r,\theta)d\theta dr\\
&\geq& 
n\omega_n+a\int_{0}^\rho \int_{\S^{n-1}}(n-1)(\sqrt a \coth{\sqrt a r})^{n-2}\left(\frac{\sinh{\sqrt{a}r}}{\sqrt a}\right)^{n-1}d\theta dr\\
&\geq&n\omega_n+n\omega_{n}\int_{0}^\rho (n-1)\sqrt a(\cosh{\sqrt ar})^{n-2}\sinh{\sqrt a r}~dr\\
&=&n\omega_{n}(\cosh{\sqrt a \rho})^{n-1}\\
&=&\mathcal{G}(\partial B^{a}_\rho),
\end{eqnarray*}
as desired. If equality holds, then equality holds in the first inequality of \eqref{eq:corcor:parallel}, which again yields that  $B_\rho$ is isometric to $B_\rho^a$.
\end{proof}

\section{Curvature of the Convex Hull}\label{sec:convexhull}\label{sec:CH}

For any convex hypersurface $\Gamma$ in a Cartan-Hadamard manifold $M$ and $\e>0$, the outer parallel hypersurface $\Gamma^\e:=(\widehat d_\Gamma)^{-1}(\e)$ is $\C^{1,1}$, by Lemma \ref{lem:GH}, and therefore its total curvature $\mathcal{G}(\Gamma^\e)$ is well-defined by Rademacher's theorem. We set
\be\label{eq:G}
\mathcal{G}(\Gamma):=\lim_{\e\to 0}\mathcal{G}(\Gamma^\e).
\ee
Recall that $\epsilon\mapsto \mathcal{G}(\Gamma^\e)$ is a decreasing function by Corollary \ref{cor:parallel}. Thus, as 
$\mathcal{G}(\Gamma^\e)\geq 0$, it follows that $\mathcal{G}(\Gamma)$ is well-defined.
The convex hull of a set $X\subset M$, denoted by $\conv(X)$, is the intersection of all closed convex sets in $M$ which contain $X$. We set
$$
X_0:=\partial\conv(X).
$$
Note that if $\conv(X)$ has nonempty interior, then $X_0$ is a convex hypersurface.
In this section  we  show  that the total positive curvature of a closed embedded $\C^{1,1}$ hypersurface $\Gamma$ in a Cartan-Hadamard manifold cannot be smaller than that of $\Gamma_0$ (Corollary \ref{cor:Gamma0}), following the same general approach indicated in \cite{kleiner1992}.  

For any set $X\subset\R^n$ and $p\in X$, the \emph{tangent cone} $T_p X$ of $X$ at $p$ is  the limit of all secant rays which emanate from $p$ and pass through a sequence of points of $X\setminus\{p\}$ converging to $p$. For a set $X\subset M$ and $p\in X$, the tangent cone is defined as
$$
T_p X:= T_p\big(\exp_p^{-1}(X)\big)\subset T_p M\simeq\R^n.
$$
 We say that a tangent cone is \emph{proper} if it does not fill up the entire tangent space. A set $X\subset\R^n$ is a \emph{cone} provided that  there exists a point $p\in X$ such that  for every $x\in X$ and $\lambda\geq 0$, $\lambda (x-p)\in X$. Then $p$ will be called an \emph{apex} of $X$. The following observation is proved in \cite[Prop. 1.8]{cheeger-gromoll1972}.

\begin{lemma}[\cite{cheeger-gromoll1972}]\label{lem:CG}
For any convex set $X\subset M$, and $p\in \partial X$, $T_pX$ is a proper convex cone in $T_p M$, and $\exp_p^{-1}(X)\subset T_p X$.
\end{lemma}

The last sentence in the next lemma is due to a theorem of Alexandrov \cite{alexandrov1939}, which states that semi-convex functions are twice differentiable almost everywhere \cite[Prop. 2.3.1]{cannarsa-sinestrari2004}.  Many different proofs of this result are available, e.g.  \cites{evans-gariepy2015,bangert1979,fu2011,harvey-lawson2013}; see \cite[p. 31]{schneider2014} for a survey. 

\begin{lemma}\label{lem:alexandrov}
Let $\Gamma$ be a convex hypersurface in a Riemannian manifold $M$. Then for each point $p$ of $\Gamma$ there exists a local coordinate chart $(U,\phi)$ of $M$ around $p$ such that $\phi(U\cap\Gamma)$ forms the graph of a semi-convex function $f\colon V\to\R$ for some open set $V\subset\R^{n-1}$. In particular $\Gamma$ is twice differentiable almost everywhere.
\end{lemma}
\begin{proof}
Let $U$ be a small normal neighborhood of $p$ in $M$, and set $\phi:=\exp_p^{-1}$. By Lemma \ref{lem:CG}, we may identify $T_p M$ with $\R^n$  such that $\phi(\Gamma\cap U)$ forms the graph of some function $f\colon V\subset\R^{n-1}\to\R$ with $f\geq 0$. We claim that $f$ is semiconvex. Indeed, since $\Gamma$ is convex, through each point $q\in\Gamma\cap U$ there passes a sphere $S_q$ of radius $r$, for some fixed $r>0$, which lies outside the domain $\Omega$ bounded by $\Gamma$. The images of small open neighborhoods of $q$  in $S_q$ under $f$ yield $\C^2$ functions $g_q\colon V_q\to\R$ which support the graph of $f$ from below in a neighborhood $V_q$ of $x_q:=f^{-1}(q)\in V$. Note that the Hessian of $g_q$ at $x_q$ depends continuously on $q$. So it follows that the second symmetric derivatives of $f$ are uniformly bounded below, i.e.,
$$
f(x+h)+f(x-h)-2f(x)\geq C|h|^2,
$$
for all $x$ in $V$, where  $C:=\sup_q|\nabla^2g_q(x_q)|$. Thus $f$ is semiconvex \cite[Prop. 1.1.3]{cannarsa-sinestrari2004}.
\end{proof}

Using Lemma \ref{lem:CG}, together with a local characterization of convex sets in Riemannian manifolds \cites{karcher1968,alexander1978}, we next establish:

\begin{lemma}\label{lem:segment}
Let $ X$ be a compact set in a Cartan-Hadamard manifold $M$, and $p\in X_0\setminus X$ be a twice differentiable point. Then the curvature of $X_0$ vanishes at $p$.
\end{lemma}
\begin{proof}
Let $\arc{\conv(X)}:=\exp_p^{-1}( \conv(X))$. By Lemma \ref{lem:CG}, $\arc{\conv(X)}\subset T_p\conv(X)$, and $T_p \conv(X)$ is a proper convex cone in $T_p M$. Thus there exists a hyperplane $H$ in $T_pM$ which passes through $p$ and with respect to which $\arc{\conv(X)}$ lies on one side. 
Next note that $H\cap \arc{\conv(X)}$ is star-shaped about $p$. Indeed if $q\in H\cap \arc{\conv(X)}$, then the line segment $pq$ in $H$ is mapped by  $\exp_p$ to a geodesic segment in $M$ which has to lie in $\conv(X)$, since $\conv(X)$ is convex. Consequently $pq$ lies in $\arc{\conv(X)}$ as desired. Now if $H\cap \arc{\conv(X)}$ contains more than one point, then there exists a geodesic segment of $M$ on $X_0$ with an end point at $p$, which forces the curvature at $p$ to vanish and we are done. So we may assume that $H\cap \arc{\conv(X)}=\{p\}$. Suppose towards a contradiction that the curvature of $X_0$ at $p$ is positive. Then there exists 
a sphere $\arc S$ in $T_p M$ which passes through $p$ and contains $\arc{\conv(X)}$ in the interior of the ball that it bounds. Let $S:=\exp_p(\arc S)$. Then $ X$ lies in the interior of the compact region bounded by $S$ in $M$. 
It is a basic fact that the second fundamental forms of $S$ and $\arc{S}$ coincide at $p$, since the covariant derivative depends only on the first derivatives of the metric. In particular $S$ has positive curvature on the closure of a neighborhood $U$ of $p$, since $\arc S$ has positive curvature at $p$. Let $S_\epsilon$ denote the inner parallel hypersurface of $S$ at distance $\epsilon$, and $U_\epsilon$ be the image of $U$ in $S_\epsilon$. Then $p$ will not be contained in $S_\epsilon$, but we may choose $\epsilon>0$ so small that $S_\epsilon$ still contains $X$, $U_\epsilon$ has positive curvature, and $S_\epsilon$ intersects $\conv (X)$ only at points of $U_\epsilon$. Let $Y$ be the intersection of the compact region bounded by $S_\epsilon$ with $\conv( X)$. Then interior of $Y$ is a locally convex set in $M$, as defined in \cite{alexander1978}. Consequently $Y$ is a convex set by a result of Karcher \cite{karcher1968}, see \cite[Prop. 1]{alexander1978}. 
So we have constructed a closed convex set in $M$ which contains $X$ but not $p$, which yields the desired contradiction, because $p\in\conv(X)$.
\end{proof}

Lemmas \ref{lem:segment} and \ref{lem:alexandrov} now indicate that the curvature of $X_0\setminus X$ vanishes almost everywhere; however, in the absence of $\C^{1,1}$ regularity for $X_0$, this information is of little use as far as proving Corollary \ref{cor:Gamma0} is concerned; see \cite{schneider2015} for a survey of curvature properties of convex hypersurfaces with low regularity, and also \cite{lytchak-petrunin} for a relevant recent result. 
We say that a geodesic segment $\alpha\colon [0,a]\to M$ is \emph{perpendicular} to a convex set $X$ provided that $\alpha(0)\in\partial X$ and $\langle\alpha'(0),x-\alpha(0)\rangle\leq 0$ for all $x\in T_{\alpha(0)} X$. We call $\alpha'(0)$ an \emph{outward normal} of $X$ at $\alpha(0)$. The following observation is well-known, see \cite[Lem. 3.2]{bishop-oniel1969}.

\begin{lemma}[\cite{bishop-oniel1969}]\label{lem:outward-geodesic}
Let $X$ be a convex set in a Cartan-Hadamard manifold $M$. Then geodesic segments which are perpendicular to $X$ at distinct points never intersect.
\end{lemma}

For every point $p\in\Gamma$ and outward unit normal $\nu\in T_p M$ we set 
$$
p^\epsilon=p_\nu^\epsilon:=\exp_p(\e \nu),
$$
and let $\Gamma^\e$ denote the outer parallel hypersurface of $\Gamma$ at distance $\e$. Furthermore, let $\kappa_i(p^\e)$ denote the principal curvatures of $\Gamma_\e$ at $p^\e$, indexed in some way.
In the next lemma we use a $2$-jet approximation result from viscosity theory \cite{fleming-soner2006}.
 
\begin{lemma}\label{lem:riccati}
Let $\Gamma$ be a convex hypersurface in a Cartan-Hadamard manifold. Suppose that $p^\epsilon$ is a twice differentiable point of the outer parallel surface $\Gamma^\e$ for some $p\in\Gamma$, outward normal $\nu$ of $\Gamma$ at $p$, and $\epsilon\geq 0$. Then $p^\epsilon$ is a twice differentiable point of $\Gamma^\e$ for all $\epsilon>0$. The principal curvatures of $\Gamma^\e$ at $p^\e$ may be indexed so that the mappings $\e\mapsto\kappa_i(p^\e)$ are $\C^1$ on $(0,\infty)$. Furthermore, if $p$ is a twice differentiable point of $\Gamma$, then $\e\mapsto\kappa_i(p^\e)$ are $\C^1$ on $[0,\infty)$. 
 \end{lemma}
\begin{proof}
Let $I:=(0,\infty)$, or $[0,\infty)$ depending on whether or not $p$ is a twice differentiable point of $\Gamma$.
Suppose that $p^\e$ is a twice differentiable point of $\Gamma^\e$, for some fixed $\e\in I$. Then
we may construct via  normal coordinates and \cite[Lem. 4.1, p. 211]{fleming-soner2006}, a  pair of $\C^2$  hypersurfaces $S_\pm$  in $M$ which pass through $p^\e$, lie on either side of  $\Gamma^\e$,  and have the same shape operator as $\Gamma^\e$ at $p^\e$, 
\be\label{eq:pe}
\mathcal{S}_{S_+}(p^\e)=\mathcal{S}_{\Gamma^\e}(p^\e)=\mathcal{S}_{S_-}(p^\e).
\ee
Since $S_\pm$ are $\C^2$, their distance functions are $\C^2$ in an open neighborhood of $p^\e$, by Lemma \ref{lem:MM}. Let $S_\pm^\delta$ denote the parallel hypersurfaces of $S_\pm$ at oriented distance 
$\delta\geq -\epsilon$. Here if $\delta>0$ we let $S_\pm^\delta$ be the \emph{outer} parallel surfaces, i.e., those which lie on the sides of $S_\pm$ where the outward normal of $\Gamma^\e$ points. If on the other hand, $\delta<0$, we let $S_\pm^\delta$ be the \emph{inner} parallel surfaces of $S_\pm$. Since the distance functions of $S_\pm$ are $\C^2$ near $S_\pm$, it follows that $S_\pm^\delta$ are $\C^2$ hypersurfaces  for $\delta$ close to $0$. Furthermore, by Ricatti's equation \cite[Cor. 3.3]{gray2004}, the shape operators $\mathcal{S}_{S_\pm^\delta}$  are determined by the initial conditions $\mathcal{S}_{S_\pm}$. Thus \eqref{eq:pe} implies that
$$
\mathcal{S}_{S_{+}^\delta}(p^{\e+\delta})=\mathcal{S}_{S_{-}^\delta}(p^{\e+\delta}).
$$
This yields that $p^{\e+\delta}$ is a twice differentiable point of $\Gamma^{\e+\delta}$ for $\delta$ sufficiently small, since $S_{\pm}^\delta$ support $\Gamma^{\e+\delta}$ on either side of $p^{\e+\delta}$. Consequently, $\epsilon\mapsto \mathcal{S}_{\Gamma^\e}(p^\e)$ is $\C^1$. Now since the shape operator is self-adjoint, it follows from a result of Rellich \cite{rellich1969}, see \cite[Chap 2, Thm. 6.8]{kato1976}, that its eigenvalues may be indexed so that they are $\C^1$ as functions of $\e$. We conclude then that the set $A\subset I$ of distances $\epsilon$ for which the conclusions of the lemma hold are open.

It remains to show that $A$ is closed. To this end let $p^{\epsilon_i}$ be twice differentiable points of $\Gamma^{\epsilon_i}$ for a sequence $\epsilon_i\in A$ converging to $\epsilon\in I$. If $\epsilon=0$, then by the above argument $[\epsilon, \epsilon+\delta)\subset A$ for some $\delta>0$ and we are done. So we may suppose that $\epsilon>0$.
Then principal curvatures of $\Gamma^{\epsilon_i}$ are uniformly bounded above, since a ball of radius $\epsilon/2$ rolls freely inside 
$\Gamma^{\epsilon_i}$ (for $i$ sufficiently large). Of course these principal curvatures are uniformly bounded below as well, since $\Gamma^{\epsilon_i}$ are convex. Now let $(S_\pm)_i$ be a pair of $\C^2$ local support surfaces of $\Gamma^{\e_i}$ as we had described above. We may assume that principal curvatures of $(S_\pm)_i$ are uniformly bounded. Then there exits $\delta>0$ independent of $i$ such that the distance functions of each $(S_\pm)_i$ are $\C^2$ on a $\delta$-neighborhood of $p^{\epsilon_i}$. Choose $i$ so large that $|\epsilon_i-\epsilon|<\delta$. Then there exist parallel surfaces $S_{\pm}$ of $(S_\pm)_i$ which are $\C^2$ and locally support $\Gamma^\epsilon$ near $p^\epsilon$. So  $p^\epsilon$ is a $\C^2$ point of $\Gamma^\e$ as desired. Finally,  the regularity property of principal curvatures near $\epsilon$ follow as described earlier.
\end{proof}

The next observation is contained essentially in Kleiner's work \cite[p. 42--43]{kleiner1992}. Here we employ the above lemmas to give a more detailed treatment as follows:

\begin{proposition}[\cite{kleiner1992}] \label{prop:Gamma0} 
Let $X$ be a compact set in a Cartan-Hadamard manifold M. Suppose that $\conv(X)$ has nonempty interior, and there exists an open neighborhood $U$ of $X_0$ in $M$ such that $X\cap U$ is a $\C^{1,1}$ hypersurface. Then
$$
\mathcal{G}(X\cap  X_0)=  \mathcal{G}( X_0).
$$
\end{proposition}
\begin{proof}
 For any set $A\subset X_0$, we define $A^\e$ as the collection of all points $p^\epsilon=p_\nu^\e=\exp_p(\e\nu)$ such that $p\in A$ and $\nu\in T_p M$ is an outward unit normal of $X_0$ at $p$.
By Lemma  \ref{lem:outward-geodesic}, 
$$
\mathcal{G}(X_0^\e)= 
\mathcal{G}\big((X_0\setminus X)^\e\big)+\mathcal{G}\big((X_0\cap X)^\e\big).
$$
As $\e\to 0$, $\mathcal{G}(X_0^\e)\to\mathcal{G}(X_0)$ by definition \eqref{eq:G}. So it suffices to show that 
$$
\mathcal{G}\big((X_0\setminus X\big)^\e)\to 0,\quad\quad\text{and}\quad\quad\mathcal{G}\big((X_0\cap X)^\e\big)\to\mathcal{G}(X_0\cap X).
$$
First we check that $\mathcal{G}((X_0\setminus X)^\e)\to 0$ (which corresponds to claim $(++)$  in \cite[p. 42]{kleiner1992}). To this end, following Kleiner \cite[p. 42]{kleiner1992}, we 
fix some $\ol\e>0$, set $\ol p:=p^{\ol \e}$, and   for all $\e\in[0,{\ol\e}]$ let 
$$
r^{\e}\colon X_0^{{\ol\e}}\to X_0^\e
$$ 
be the (nearest point) projection $\ol p\mapsto p^\e$ (which is a Lipschitz map).  In particular note that $p^\e=r^{\e}(\ol p)$. Set
$
J(\e):=\textup{Jac}_{\ol{p}}(r^{\e}).
$
Then, for all $\e\in[0,{\ol\e}]$, 
\be\label{eq:kleiner}
\mathcal{G}\big((X_0\setminus X)^\e\big)
=
\int_{\ol p\in(X_0\setminus X)^{\ol\e}} GK(\e)J(\e)d\sigma,
\ee
where $GK(\e):=GK_{X_0^\e}(p^\e)$. 
So to show that $\mathcal{G}((X_0\setminus X)^\e)\to 0$ it suffices to check, via the dominated convergence theorem, that for almost all $\ol p\in (X_0\setminus X)^{\ol\e}$, 
\begin{itemize}
\item[(I)]{$GK(\e)J(\e)\leq C$, for $0<\epsilon\leq\ol\epsilon$, and}
\item[(II)]{$GK(\e)J(\e)\to 0$, as $\epsilon\to 0$.}
\end{itemize}
The above claims correspond to items $1$ and $2$ in \cite[p. 42]{kleiner1992}.

To establish (I) note that, at every twice differentiable point $\ol p\in (X_0\setminus X)^{\ol\e}$,  the second fundamental form of $X_0^{\ol\e}$ is bounded above,  since $X_0^{\ol\e}$ is supported from below by balls of radius $\ol\e$ at each point. As discussed in \cite[p. 42--43]{kleiner1992}, (I) then follows via an argument using Jacobi's equation. For the convenience of the reader, we provide an alternative self-contained 
proof of (I) via Riccati's equation as follows.  By \cite[Thm. 3.11]{gray2004}, 
\be\label{eq:J-prime}
J^{\prime}(\e)=(n-1)H(\e)J(\e),
\ee
where $H(\e):=H_{X_0^{\e}}(p^\e)\geq 0$ is the mean curvature of $X_0^{\e}$ at $p^\e$ (recall that, as we pointed out in Section \ref{sec:d-convex}, the sign of our mean curvature is opposite to that in \cite{gray2004}).
Let $\kappa_i(\e):=\kappa_i(p^\e)$ be an indexing of the principal curvatures as in Lemma \ref{lem:riccati}.
By Riccati's equation for parallel hypersurfaces \cite[Cor. 3.5]{gray2004}, if $\kappa_i(\e)$ are distinct,  we have
\be\label{eq:kappa-i-prime}
\kappa_i'(\epsilon)=-\kappa_i^2(\epsilon)-R_{i n in}(\epsilon),
\ee
where $R_{i n in}(\e)$ denotes the sectional curvature of $M$ at $p^\e$, with respect to the plane generated by a principal direction of $X_0^\e$ and its normal vector. It follows that
\be\label{eq:GK-prime-e}
GK^{\prime}(\e)=-\left((n-1)H(\e)+\textup{Ric}(\e) \sum_{i=1}^{n-1} \frac{1}{\kappa_i(\e)}\right) GK(\e)\geq -(n-1)H(\e)GK(\e),
\ee
where $\textup{Ric}(\e)$ denotes the Ricci curvature of $M$ at $p^\e$ with respect to the direction of the geodesic which connects $p$ to $p^\e$. Since the inequality in \eqref{eq:GK-prime-e} holds when $\kappa_i(\e)$ are distinct, and $GK$, $H$ are $\C^1$, it follows that \eqref{eq:GK-prime-e} holds in general.
So we have 
$$
\big(GK(\e) J(\e)\big)'\geq -(n-1)H(\e)GK(\e)J(\e)+GK(\e)(n-1)H(\e)J(\e)=0.
$$
Note that $J(\ol\e)=1$, since $r^{\ol\e}$ is the identity map.
Hence, for $\e\leq\ol\e$,
\be\label{eq:GKeJe}
GK(\e) J(\e)\leq GK(\ol \e)J(\ol\e)=GK(\ol \e).
\ee 
But $GK(\ol \e)$ is uniformly bounded above, since as we had mentioned earlier, a ball  of radius $\ol\e$ rolls freely inside $X_0^{\ol\e}$.  So we obtain (I).

To see (II) let us first consider the special case where $p=r^{\ol \e}(p^{\ol\e})$ is a twice differentiable point of $X_0\setminus X$ (which is almost every point of $X_0\setminus X$ by Lemma \ref{lem:alexandrov}). Then, by Lemma \ref{lem:riccati}, $p^\epsilon$ is a twice differentiable point of $X_0^\epsilon$ for all $\e\in[0,{\ol\e}]$. Furthermore,
$
 J(\e)\leq 1,
$
since  in a Hadamard space projection into convex sets is nonexpansive  \cite[Cor. 2.5]{bridson-haefliger1999}.
So it suffices to show that
$
GK(\e)\to 0,
$
which is indeed the case by Lemmas \ref{lem:segment} and  \ref{lem:riccati}.
In the absence of a priori knowledge that $X_0\setminus X$ is $\C^{1,1}$, however, we do not know whether $p$ is a twice differentiable point of $X_0\setminus X$ for almost all $\ol p\in (X_0\setminus X)^{\ol\e}$. So a more general approach is needed to establish (II).

The argument for establishing (II) in general is as follows. Let $\ol p$ be a twice differentiable point of $(X_0\setminus X)^{\ol\e}$. Then $p^\e$ will be a twice differentiable point of $(X_0\setminus X)^{\e}$ for all $\e\in(0,\ol\e]$. Let $\kappa_{i}(\epsilon)$ denote the  the principal curvatures  of $(X_0\setminus X)^{\e}$ at $p^\epsilon$. As observed in \cite[p. 43]{kleiner1992}, to establish (II) it suffices to show that $\inf_i\kappa_{i}(\epsilon)\to 0$ as $\epsilon\to 0$. An alternative reasoning to that presented in \cite[p. 43]{kleiner1992} is as follows. Set $GK_i(\e):=GK(\e)/\kappa_i(\epsilon)$. Then using \eqref{eq:kappa-i-prime} and \eqref{eq:GK-prime-e} we obtain
$$
GK_i'(\e)\geq -\left((n-1)H(\e)-\kappa_i(\e)+\frac{R_{inin}(\e)}{\kappa_i(\e)}\right)GK_i(\e)\geq -(n-1)H(\e)GK_i(\e).
$$
This together with \eqref{eq:J-prime} yields that $(GK_i(\e)J(\e))'\geq 0$. So $GK_i(\e)J(\e)\leq GK_i(\ol\e)J(\ol\e)\leq GK_i(\ol\e)\leq C$, since as we had argued earlier, $J(\e)\leq 1$ and principal curvatures of $X_0^{\ol\e}$ are bounded above. So it follows that
$$
GK(\e)J(\e)=\kappa_{i}(\e)GK_i(\e)J(\e)\leq C\kappa_{i}(\e),
$$
for all $0\leq i\leq n-1$.
In particular, if $\inf_i\kappa_{i}(\e)$ vanishes, as $\epsilon\to 0$, then so does $GK(\e)J(\e)$ and we obtain (II) as claimed.

Now it remains to establish that $\inf_i\kappa_{i}(\epsilon)\to 0$ as $\epsilon\to 0$.
In \cite[p. 43]{kleiner1992} it is stated that  this holds because $p=r^{\ol\e}(\ol p)\in X\setminus X_0$; however, the reasoning is not  mentioned. Here we include a reasoning. Let $S$ be a positively curved surface with boundary which contains $p$ in its interior, is orthogonal to the outward normal $\nu$ of $X_0$ at $p$ with $\exp(\e\nu)=\ol p$, and its mean curvature vector is parallel to $-\nu$ (So $S$ curves towards $X_0$). We may construct $S$  by taking a portion of the image under the exponential map of a large sphere in $T_p M$ which passes through $p$ and is orthogonal to $\nu$. In particular note that principal curvatures of $S$ may be arbitrarily small.
Next note that $S$ must enter the interior of $\conv(X)$; otherwise, after replacing $S$ with a surface with smaller curvature, we can make sure that $\partial S$ is disjoint from $\conv(X)$. Then, as described in Lemma \ref{lem:segment}, by pushing $S$ a small distance towards $-\nu$, we may replace $\conv(X)$ with a smaller convex set containing $X$, which is not possible. Now let $S^\epsilon$ be the outward parallel surface of $S$. For $\epsilon$ small, $S^\epsilon$ will remain positively curved. Furthermore, since $S$ always has a point in the interior of $\conv(X)$, it follows that $S^\epsilon$ always has a point inside the convex set bounded by $X_0^\epsilon$.
Hence $\inf_i\kappa_{i}(\epsilon)$ cannot be larger than all the principal curvatures of $S^\epsilon$ at $p^\epsilon$. But principal curvatures of $S$ may be arbitrarily small. So the principal curvatures of $S^\epsilon$ will also be arbitrarily small, for small $\epsilon$. Hence $\inf_i\kappa_{i}(\epsilon)$ must vanish as claimed, which completes the proof of (II).

It remains to show then that $\mathcal{G}\big((X_0\cap X)^\e\big)\to\mathcal{G}(X_0\cap X)$.
To see this note that $GK(\e)J(\e)\to GK(0)J(0)$ by Lemma \ref{lem:riccati}. Then the dominated convergence theorem, as we argued above, completes the proof.
\end{proof}

Finally we arrive at the main result of this section. The \emph{total positive curvature} of a closed $\C^{1,1}$ embedded hypersurface $\Gamma$ is defined as
$$
\mathcal{G}_+(\Gamma):=\int_{\Gamma_+}GK\,d\sigma,
$$
where $\Gamma_+\subset\Gamma$ is the region where $GK\geq 0$.
The study of total positive curvature goes back to Alexandrov \cite{alexandrov:tight} and Nirenberg \cite{nirenberg:tight}, and its relation to isoperimetric problems  has been well-known \cites{cgr,cgr:relative}. The  minimizers for this quantity, which are called \emph{tight hypersurfaces}, have been extensively studied since Chern-Lashof \cites{chern-lashof:tight1,chern-lashof:tight2}; see \cite{cecil-chern:book} for a survey

\begin{corollary}\label{cor:Gamma0}
Let $\Gamma$ be a closed $\C^{1,1}$ hypersurface embedded in a Cartan-Hadamard manifold. Then 
$$
\mathcal{G}_+(\Gamma)\geq \mathcal{G}(\Gamma_0).
$$
\end{corollary}
\begin{proof}
Note that $\mathcal{G}_+(\Gamma)\geq \mathcal{G}_+(\Gamma\cap\Gamma_0)$. Furthermore, since $\Gamma$ is supported by $\Gamma_0$ from above,  $GK_\Gamma(p)\geq GK_{\Gamma_0}(p)\geq0$ for all  twice differentiable points $p\in\Gamma\cap\Gamma_0$. 
Hence $\mathcal{G}_+(\Gamma\cap\Gamma_0)=\mathcal{G}(\Gamma\cap\Gamma_0)$. Finally, $\mathcal{G}(\Gamma\cap\Gamma_0)=\mathcal{G}(\Gamma_0)$ by Proposition \ref{prop:Gamma0}, which completes the proof.
\end{proof}

\begin{note}
Proposition \ref{prop:Gamma0} would follow immediately from Lemma \ref{lem:segment}, if we could show that $X_0$ is $\C^{1,1}$,
whenever $X$ is $\C^{1,1}$ near $X_0$. Here we show that the latter holds for closed $\C^{1,1}$ hypersurfaces $\Gamma$ bounding a domain $\Omega$ in $\R^n$.
Indeed, when $\Gamma$ is $\C^{1,1}$, there exists $\epsilon>0$ such that the inner parallel hypersurface $\Gamma^{-\e}$, obtained by moving a distance $\e$ along inward normals is embedded, by Lemma \ref{lem:GH}. Let $D\subset\Omega$ be the domain bounded by $\Gamma^{-\e}$.  We claim that
\be\label{eq:convDe}
\big(\conv(D)\big)^\epsilon=\conv(D^\epsilon),
\ee
where $(\cdot)^\e$ denotes the outer parallel hypersurface.
This shows that a ball (of radius $\e$) rolls freely inside $\conv(D^\epsilon)$. Thus $(D^\epsilon)_0$ is $\C^{1,1}$ by Lemma \ref{lem:GH} which completes the proof, since $D^\e=\Omega$. So
$(D^\epsilon)_0=\Omega_0=\Gamma_0$.
To prove \eqref{eq:convDe} note that, since $D\subset\conv(D)$, we have
$D^\epsilon\subset(\conv(D))^\epsilon$, which in turn yields
$$
\conv(D^\epsilon)\subset\conv\big((\conv(D))^\epsilon\big)=\big(\conv(D)\big)^\epsilon,
$$
since outer parallel hypersurfaces of a convex hypersurface are convex in any Cartan-Hadamard manifold (Lemma \ref{lem:u-convex}).
To establish the reverse inclusion, suppose that $p\not\in \conv(D^\epsilon)$. Then there exists a convex set $Y$ which contains $D^\epsilon$ but not $p$. 
Consequently the inner parallel hypersurface $Y^{-\e}$ contains $D$ and is disjoint from $B^\e(p)$, the ball of radius $\e$ centered at $p$. But $Y^{-\e}$ is convex, since the signed distance function is convex inside convex sets in nonnegatively curved manifolds \cite[Lem. 3.3 p. 211]{sakai1996}. So $\conv(D)$ is disjoint from $B^\e(p)$, which in turn yields that  $p\not\in (\conv(D))^\e$. Thus we have established that 
$$
\big(\conv(D)\big)^\e\subset \conv(D^\epsilon)
$$ 
as desired.
\end{note}

\begin{note}\label{note:monotone}
The method of part (I) in the proof of Proposition \ref{prop:Gamma0}, specifically relation \eqref{eq:GKeJe}, can be used to give an alternative proof of  monotonicty of total curvature for parallel hypersurfaces in Cartan-Hadamard manifolds, established in Corollary \ref{cor:parallel}.
\end{note}

\section{The Isoperimetric Inequality}\label{sec:CHC}

In this section we establish the link between the two problems cited in the introduction:

 \begin{theorem}\label{thm:GK-CH}
Suppose that the total curvature inequality \eqref{eq:GK} holds in a Cartan-Hadamard manifold $M$. Then the isoperimetric inequality \eqref{eq:II} holds in $M$ as well with equality only for Euclidean balls.
 \end{theorem}

The proof employs the results of the last section on convex hulls, and proceeds via the well-known isoperimetric profile argument \cites{ritore2010,ros2005} along the same general lines indicated by Kleiner \cite{kleiner1992}. The \emph{isoperimetric profile} \cites{berger2003,bbg1985} of any open subset $U$ of a Riemannian manifold $M$ is the function $\mathcal{I}_U\colon [0, \textup{vol}(U))\to\R$ given by
$$
\mathcal{I}_{U}(v):=\inf\big\{\text{per}(\Omega)\mid \Omega\subset U,\, \textup{vol}(\Omega)=v, \,\textup{diam}(\Omega)<\infty \big\},
$$
where $\emph{diam}$ is the diameter, $\emph{vol}$ denotes the Lebesgue measure, and $\emph{per}$ stands for perimeter; see \cite{giusti1984, chavel2001} for the general definition of perimeter (when $\partial\Omega$ is piecewise $\C^1$, for instance, $\textup{per}(\Omega)$ is just the $(n-1)$-dimensional Hausdorff measure of $\partial\Omega$).
Proving the isoperimetric inequality \eqref{eq:II}  is equivalent to showing that 
$$
\mathcal{I}_{M}\geq\mathcal{I}_{\R^n},
$$
 for any Cartan-Hadamard manifold $M$. To this end it suffices to show that $\mathcal{I}_B\geq \mathcal{I}_{\R^n}$ for a family of (open) geodesic balls $B\subset M$ whose radii grows arbitrarily large and  eventually covers any bounded set $\Omega\subset M$. So we fix a geodesic ball $B$ in $M$ and consider its \emph{isoperimetric regions}, i.e., sets $\Omega\subset B$ which have the least perimeter for a given volume, or satisfy $\textup{per}(\Omega)=\mathcal{I}_B(\textup{vol}(\Omega))$. The existence of these regions are well-known, and they have the following regularity properties:

\begin{lemma}[\cite{gmt1981, Stredulinsky-Ziemer1997}]\label{lem:exist}
For any $v\in (0, \textup{vol}(B))$ there exists an isoperimetric region $\Omega \subset B$ with $\textup{vol}(\Omega)=v$. Let $\Gamma:=\partial\Omega$, $H$ be the normalized mean curvature of $\Gamma$ (wherever it is defined),  and $\Gamma_0:=\partial\conv(\Gamma)$. Then
\begin{enumerate}[(i)]
\item$\Gamma\cap B$ is $\C^\infty$  except for a closed set $\textup{sing}(\Gamma)$ of Hausdorff dimension at most $n-8$. Furthermore,
$H\equiv H_0=H_0(v)$  a constant  on $\Gamma\cap B\setminus \textup{sing}(\Gamma)$.

\item $\Gamma$ is $\C^{1,1}$ within an open neighborhood $U$ of $\partial B$ in $M$. Furthermore,
$ H\leq H_0$ almost everywhere on $U\cap\Gamma$.

\item $d\big(\textup{sing}(\Gamma), \Gamma_0\big)\geq \e_0>0$.
\end{enumerate}
In particular $\Gamma$ is $\C^{1,1}$ within an open neighborhood of $\Gamma_0$ in $M$.
\end{lemma}

\begin{proof}
Part (i)  follows from Gonzalez, Massari, and Tamanini \cite{gmt1983}, and 
(ii) follows from Stredulinsky and Ziemer \cite[Thm 3.6]{Stredulinsky-Ziemer1997}, who studied the identical variational problem in $\R^n$. Indeed the $\C^{1,1}$ regularity  near $\partial B$ is based on the classical obstacle problem for graphs which extends in a straightforward way to Riemannian manifolds; see also Morgan \cite{morgan2003}.
To see (iii) note that by (i), $\text{sing}(\Gamma)$ is closed, and by (ii), $\text{sing}(\Gamma)$ lies in $B$. So it suffices to check that points $p\in \Gamma\cap \Gamma_0 \cap B$ are not singular. This is the case since $T_p\Gamma\subset T_p\conv(\Gamma)$ which is a convex subset of $T_p M$. Therefore $T_p \Gamma$ is contained in a half-space of $T_p M$ generated by any support hyperplane of $T_p\conv(\Gamma)$ at $p$. This forces $T_p \Gamma$ to be a hyperplane  \cite[Cor. 37.6]{simon1983}. Consequently $\Gamma$ will be $\C^\infty$ in a neighborhood of $p$ \cite[Thm. 5.4.6]{federer:book}, \cite[Prop. 3.5]{morgan2003}.
\end{proof}

Now let $\Omega\subset B$ be an isoperimetric region with volume $v$, as provided by Lemma \ref{lem:exist}. By Proposition  \ref{prop:Gamma0}, $\mathcal{G}(\Gamma_0) = \mathcal{G}(\Gamma\cap\Gamma_0)$. By the total curvature inequality \eqref{eq:GK}  and definition \eqref{eq:G}, $\mathcal{G}(\Gamma_0)\geq n\omega_n$. Thus we have
\be\label{eq:n-omega-n}
n\omega_n \leq \mathcal{G}(\Gamma_0) = \mathcal{G}(\Gamma\cap\Gamma_0)=\int_{\Gamma \cap \Gamma_0}GKd\sigma,
\ee
where $GK$ denotes the Gauss-Kronecker curvature of $\Gamma$.
Note that $GK\geq 0$ on $\Gamma \cap \Gamma_0$, since at these points $\Gamma$ is locally convex. So the arithmetic versus geometric means inequality yields that 
$
GK\leq H^{n-1}
$
 on $\Gamma\cap\Gamma_0$. Thus, by \eqref{eq:n-omega-n},
\begin{eqnarray} \label{proof.20}\notag
n\omega_n
&\leq& \int_{\Gamma\cap\Gamma_0} GK d\sigma\\ \notag
&\leq& \int_{\Gamma\cap \Gamma_0}H^{n-1}d\sigma \\ 
&=&\int_{\Gamma\cap \partial B}H^{n-1}d\sigma +\int_{\Gamma\cap\Gamma_0\cap B}H_0^{n-1}d\sigma\\ \notag
 &\leq&\int_{\Gamma\cap \partial B}H_0^{n-1}d\sigma +\int_{\Gamma\cap B}H_0^{n-1}d\sigma\\  \notag
&=&  H_0^{n-1}\,\text{per}(\Omega).
\end{eqnarray}
Hence it follows that
\be \label{mc.compare}
H_0\big(\textup{vol}(\Omega)\big) \;\geq\;\left(\frac{n\omega_n}{\textup{per}(\Omega)}\right)^\frac{1}{n-1}=\;\ol H_0\big(\text{per}(\Omega)\big),
\ee
where $\ol H_0(a)$ is the mean curvature of a ball of perimeter $a$ in $\R^n$.  It is well-known that $\mathcal{I}_B$ is continuous  and  increasing \cite{ritore2017}, and thus is differentiable almost everywhere. Furthermore, 
$
\mathcal{I}_B'(v)=(n-1)H_0(v)
$ 
at all differentiable points $v\in(0,\textup{vol}(B))$ \cite[Lem. 5]{hsiang1992}. Then it follows from \eqref{mc.compare}, e.g., see \cite[p. 189]{choe-ritore2007},  that $\mathcal{I}'_B\geq \mathcal{I}'_{\R^n}$ almost everywhere on $[0,\textup{vol}(B))$. Hence
\be\label{eq:IBIBe}
\mathcal{I}_B(v)\geq \mathcal{I}_{\R^n}(v),
\ee
for all $v\in[0, \textup{vol}(B))$ as desired. So we have established the isoperimetric inequality \eqref{eq:GK} for Cartan-Hadamard manifolds. It remains then to show that equality holds in  \eqref{eq:GK} only for Euclidean balls. To this end we first record that:

\begin{lemma} \label{prop:equal-case} 
Suppose that equality in \eqref{eq:II} holds for a bounded set $\Omega\subset M$.  Then $\Gamma$ is strictly convex, $\C^\infty$,  and  has  constant mean curvature $H_0$.  Furthermore, the principal curvatures of $\Gamma$ are all equal to $H_0$ .
\end{lemma}
\begin{proof}
If equality holds in \eqref{eq:II}, then we have equality in \eqref{eq:IBIBe} for some ball $B\subset M$ large enough to contain $\Omega$, and $v=\textup{vol}(\Omega)$. This in turn forces equality to hold successively in  \eqref{mc.compare}, and  \eqref{proof.20}.
Now equality between the third and fourth lines in \eqref{proof.20} yields that
\be\label{eq:Hn-1Gamma-cap-B}
\mathcal{H}^{n-1}(\Gamma\cap \partial B)=0,
\ee
\be\label{eq:Gamma=Gamma0}
\Gamma=\Gamma_0.
\ee
Then equality between the second and third lines in \eqref{proof.20} yields that
\be\label{eq:H=GK}
H^{n-1}= GK\equiv H_0^{n-1},
\ee
on $(\Gamma\cap B)\setminus \textup{sing}(\Gamma)$.
By \eqref{eq:Gamma=Gamma0}, $\Gamma$ is convex. Thus
as in the proof of part (iii) of Lemma \ref{lem:exist}, for every point $p\in \Gamma\cap B$, $T_p \Gamma$ is a hyperplane. So $\Gamma\cap B$ is  $\C^{\infty}$. On the other hand by part (ii) of Lemma \ref{lem:exist}, near $\partial B$, $\Gamma$ is locally a $\C^{1,1}$ graph and thus every point of $\Gamma$ has a H\"{o}lder continuous unit normal. Furthermore, $\Gamma$ has $H^{n-1}$ almost everywhere constant mean curvature $H_0$, by \eqref{eq:Hn-1Gamma-cap-B}. It follows  that $\Gamma $ is $\C^{\infty}$  in a neighborhood of $\partial B$; see \cite[Thm. 27.4]{maggi2012} for details of this well-known argument.
Finally \eqref{eq:H=GK} implies that all principal curvatures are equal to $H_0$ at all points of $\Gamma$.
\end{proof}

We also need the following basic fact:

\begin{lemma}\label{lem:C11converge}
Let $\Gamma_i$ be a sequence of $\C^{2}$ convex hypersurfaces in $M$ which converge to a convex hypersurface $\Gamma$ with respect to the Hausdorff distance. Suppose that the principal curvatures of $\Gamma_i$ are bounded above by a uniform constant. Then $\Gamma$ is $\C^{1,1}$.
\end{lemma}
\begin{proof}
Let $p$ a point of $M$, and set $\arc\Gamma:=\exp_p^{-1}(\Gamma)$, $\arc\Gamma_i:=\exp_p^{-1}(\Gamma_i)$. Then  $\arc\Gamma_i$ will still be $\C^{2}$, and their principal curvatures are uniformly bounded above. It follows then from Blaschke's rolling theorem \cite{howard1999} that a ball of radius $\epsilon$ rolls freely inside $\arc\Gamma_i$. Thus a ball of radius $\epsilon$ rolls freely inside $\arc\Gamma$, or $\textup{reach}(\arc\Gamma)>0$. Hence $\arc\Gamma$ is $\C^{1,1}$ by Lemma \ref{lem:GH}, which in turn yields that so is $\Gamma$.
\end{proof}

Now suppose that  equality holds in \eqref{eq:II} for some region $\Omega$ in a Cartan-Hadamard manifold $M$. Then equality holds successively in \eqref{mc.compare},   \eqref{proof.20}, and \eqref{eq:n-omega-n}. So we have $\mathcal{G}(\Gamma_0)=n\omega_n$. But we know from Lemma \ref{prop:equal-case} that $\Gamma$ is convex, or $\Gamma_0=\Gamma$. So 
\be\label{eq:GGammanomegan}
\mathcal{G}(\Gamma)=n\omega_n.
\ee 
Let $\lambda_1:=\textup{reach}(\Gamma)$, as defined in Section \ref{sec:distance}. 
Furthermore note that, by Lemma \ref{prop:equal-case}, $\Gamma$ is $\C^{\infty}$. Thus $\lambda_1>0$ by Lemma \ref{lem:GH}. Set $u:=\widehat{d}_\Gamma$. Then $\Gamma_\lambda:=u^{-1}(-\lambda)$ will be a $\C^{\infty}$ hypersurface for $\lambda\in[0,\lambda_1)$ by Lemma \ref{lem:MM}. 
 For any point $p$ of $\Gamma$,  let $p_\lambda$ be the point obtained by moving $p$ the distance of $\lambda$ along the inward geodesic orthogonal to $\Gamma$ at $p$, and set 
 $R_{\ell n \ell n}(\lambda):=R_{\ell n \ell n}(p_\lambda)$. We claim that
 \be\label{eq:Rlambda}
 R_{\ell n \ell n}(\lambda)\equiv0,
\ee
for $\lambda\in[0,\lambda_1]$.
 To see this note that for $\lambda$ sufficiently small $\Gamma_\lambda$ is positively curved by continuity. Let $\ol\lambda$ be the supremum of $x<\lambda_1$ such that $\Gamma_\lambda$ is positively curved on $[0,x)$. By  assumption \eqref{eq:GK}, $\mathcal{G}(\Gamma_{\lambda})\geq n\omega_n$.
 Thus, by \eqref{eq:GGammanomegan} and Corollary \ref{cor:parallel},
 $$
0\geq n\omega_n- \lim_{\lambda\to\ol\lambda}\mathcal{G}(\Gamma_{\lambda})=\mathcal{G}(\Gamma)- \lim_{\lambda\to\ol\lambda}\mathcal{G}(\Gamma_{\lambda})=-\int_{\Omega\setminus D_{\ol\lambda}}
R_{\ell n\ell n} \frac{GK}{\kappa_\ell}\,d\mu\geq 0,
 $$
 where $D_{\ol\lambda}$ is the limit of the regions bounded by $\Gamma_\lambda$ as $\lambda\to\ol\lambda$.
 So $R_{\ell n\ell n}(\lambda)\equiv 0$ for $\lambda<\ol\lambda$.
By Riccati's equation  \cite[Cor. 3.3]{gray2004}, it follows that
$$
S'(\lambda)=S^2(\lambda),
$$
for $\lambda<\ol\lambda$,
where $S(\lambda)$ is the shape operator of $\Gamma_\lambda$ at $p_\lambda$.
By Lemma \ref{prop:equal-case}, $S(0)=H_0 I$, where $I$ is the identity transformation. Thus, solving the differential equation above yields that $S(\lambda)=H_\lambda I$ where
$$
H_\lambda:=\frac{H_0}{1-\lambda H_0},
$$
for $\lambda<\ol\lambda$. Now suppose that $\ol\lambda<\lambda_1$. Then $\Gamma_{\ol\lambda}$ will be a $\C^2$ hypersurface, and therefore, by continuity, it will have constant principal curvature $H_{\ol\lambda}:=\lim_{\lambda\to\ol\lambda} H_\lambda$. Since $\Gamma_{\ol\lambda}$ is a closed hypersurface, $H_{\ol\lambda}>0$. So $\Gamma_{\ol\lambda}$
has positive curvature, which is not possible if $\ol\lambda<\lambda_1$. Thus we conclude that $\ol\lambda=\lambda_1$, which establishes \eqref{eq:Rlambda} as claimed. 
 
 Next note that if the principal curvatures of $\Gamma_\lambda$ remain uniformly bounded above, for $\lambda<\lambda_1$, then $\Gamma_{\lambda_1}$ is a $\C^{1,1}$ hypersurface by Lemma \ref{lem:C11converge}, which is not possible, since $\lambda_1=\textup{reach}(\Gamma)$. So some principal curvature of $\Gamma_\lambda$ must blow up, as $\lambda\to\lambda_1$. But $\Gamma_\lambda$ has constant principal curvatures. Thus all principal curvatures of $\Gamma_\lambda$ blow up. By Gauss's equation, then all sectional curvatures of $\Gamma_\lambda$ blow up. Consequently, by Bonnet-Myers theorem, diameter of $\Gamma_\lambda$ converges to zero. In other words, $\Gamma_\lambda$ collapses to a point, say $x_0$, as $\lambda\to\lambda_1$. 
So $\Gamma=\partial B_{\lambda_1}$ or $\Omega=B_{\lambda_1}$, a geodesic ball of radius $\lambda_1$ centered at $x_0$. Furthermore, the condition $R_{\ell n \ell n}=0$ now means that  along each geodesic segment which connects $x_0$ to $\partial B_{\lambda_1}$, the sectional curvatures of $M$ with respect to the planes tangent to that geodesic vanish. Thus by Lemma \ref{lem:radial-curvature} all sectional curvatures of $B_{\lambda_1}$ vanish. So $\Omega$ is a Euclidean ball as claimed.

\appendix
 \section{Smoothing the Distance Function}\label{sec:inf-convolution}
 
In this section we discuss how to smooth the (signed) distance function $\widehat{d}_\Gamma$ of a hypersurface $\Gamma$ in a Riemannian manifold $M$ via inf-convolution. We also derive some basic estimates for the derivatives of the smoothing via the associated proximal maps. For $t>0$, the
\emph{inf-convolution} (or more precisely  \emph{Moreau envelope} or \emph{Moreau-Yosida regularization}) of a function $u\colon M\to\R$ is given by
\be \label{eq:infcov-def}
\tilde u^{\,t}(x):=\inf_y\left\{ u(y)+\frac{d^2(x,y)}{2t}\right\}.
\ee
It is well-known that $\tilde u^{\,t}$ is the unique viscosity solution of the Hamilton-Jacobi equation $f_t+(1/2)|\nabla f|^2=0$ for functions $f\colon \R\times M\to\R$ satisfying the initial condition $f(0,x)=u(x)$. Furthermore, when $M=\R^n$, $\tilde u^{\,t}$ is characterized by the fact that its epigraph is the Minkowski sum of the epigraphs of $u$ and $|\cdot|^2/(2t)$ \cite[Thm. 1.6.17]{schneider2014}. 
The following properties are well-known,
\be\label{eq:semigroup}
\tilde{(\tilde u^{\,t})}^s=\tilde u^{\,t+s},\quad\quad\;\;\text{and}\quad\quad\;\; \tilde{\lambda u}^{\,t}=\lambda \tilde u^{\lambda t},
\ee
e.g., see \cite[Prop. 12.22]{bauschke-combettes2017}. A simple but highly illustrative example of inf-convolution occurs when it is applied to $\rho(x):=d(x_0,x)$, the distance from a single point $x_0\in M$. Then  
\be\label{eq:tilde-rho}
\tilde\rho^{\,t}(x)=
\left \{ 
\begin{array}{ll}
\rho^2(x)/(2t), &\,\, \text{if $\rho(x) \leq t$},\\
\rho(x)-t/2, & \,\,\text{if $\rho(x)>t$} ,
\end{array}
\right.
\ee
which is known as the \emph{Huber function}; see Figure \ref{fig:huber} which shows the graph of $\tilde\rho^{\,t}$ when $M=\R$ and $x_0=0$.
\begin{figure}[h]
\centering
\begin{overpic}[height=1.25in]{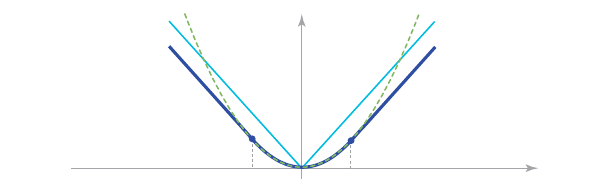}
\put( 58.25, 0.25){\Small $t$}
\put( 39, 0.25){\Small $-t$}
\put( 73, 27.5){\Small $\rho$}
\put( 73, 22.5){\Small $\tilde\rho^{\,t}$}
\put( 65.5, 28){\Small $\frac{\rho^{2}}{2t}$}
\put( 49.25, -0.5){\Small $x_0$}
\put( 88, 0.5){\Small $x$}
\end{overpic}
\caption{}\label{fig:huber}
\end{figure}
 Note that $\tilde\rho^{\,t}$ is $\C^{1,1}$ and convex, $\inf (\tilde\rho^{\,t})=\inf(\rho)$,  $|\nabla \tilde\rho^{\,t}|\leq 1$ everywhere, $|\nabla \tilde\rho^{\,t}|=1$ when $\rho>t$, and $|\nabla^2\tilde\rho^{\,t}|\leq C/t$. Remarkably enough, all these properties are shared by the inf-convolution of $\widehat{d}_\Gamma$ when $\Gamma$ is $d$-convex, as we demonstrate below.

Some of the following observations are well-known or easy to establish in $\R^n$ or even Hilbert spaces \cite{bauschke-combettes2017,cannarsa-sinestrari2004}.  In the absence of a linear structure, however, finer methods are required to examine the inf-convolution on Riemannian manifolds, especially with regard to its differential properties \cites{fathi2003, bernard2007,azagra-ferrera2006,azagra-ferrera2015,bacak2014}. First let us record that, by \cite[Cor. 4.5]{azagra-ferrera2006}:

\begin{lemma}[\cites{azagra-ferrera2006}]\label{lem:infconv0}
Let $u$ be a convex function on a Cartan-Hadamard manifold. Then for all $t>0$ the following properties hold:
\begin{enumerate}[(i)]
\item{$\tilde u^{\,t}$ is $\C^1$ and convex.}
\item{$t\mapsto\tilde u^{\,t}(x)$
is nonincreasing, and 
$
\lim_{t\goto 0} \tilde u^{\,t}(x)=u(x).
$}
\item{$\textup{inf}(\tilde u^{\,t})=\textup{inf}(u)$, and minimum points of $\tilde{u}^{\,t}$ coincide with those of $u$.}
\end{enumerate}
\end{lemma}

See also \cite[Ex. 2.8]{bacak2014} for part (i) above.  Next let us rewrite \eqref{eq:infcov-def} as
$$
\tilde u^{\,t}(x)=\inf_y F(y),  \quad\quad\quad F(y)=F(x,y):=u(y) +\frac{d^2(x,y)}{2t}.
$$
Since $d^2(x,y)$ is strongly convex and $u$ is convex, $F(y)$ is strongly convex and thus its infimum is achieved at a unique point 
$$
x^*:=\textup{prox}^u_{t}(x),
$$
which is called the \emph{proximal point} \cite{bauschke-combettes2017} or \emph{resolvent} \cite{bacak2014} of $\tilde u^{\,t}$ at $x$. In other words,
 $$
 \tilde u^{\,t}(x)=F\big(x^*). 
 $$ 
  
The next estimate had been observed earlier \cite[Prop. 2.1]{azagra-ferrera2015} for $2tL$.
\begin{lemma} \label{lem:infconv2} 
Let $u$ be an  $L$-Lipschitz function on a Riemannian manifold. Then 
$$
d(x, x^*)\leq tL.
$$
\end{lemma}

\begin{proof} 
Suppose, towards a contradiction, that $d(x^*, x)>tL$. Then there exists an $\epsilon>0$ such that
$$
d(x^*, x)\geq (1+\e)tL.
$$
 Choose a point $x'$ on the geodesic segment between $x$ and $x^*$ with 
 $$
 d(x, x')=d(x,x^*)-\e tL.
 $$
  Since $\epsilon$ may be chosen arbitrarily small, we may assume that $x'$ is arbitrarily close to $x^*$. Thus by the local $L$-Lipschitz assumption, $u(x')-u(x^*)\leq Ld(x^*,x')$. Consequently,
\begin{eqnarray*}
 F(x')-F(x^*)
 &=& u(x')-u(x^*)+\frac{d^2(x, x')-d^2(x, x^*)}{2t}\\
&\leq& Ld(x^*,x')+ \frac{ (d(x,x')-d(x, x^*))(d(x,x')+d(x, x^*))}{2t}\\
&\leq& Ld(x^*,x')-d(x^*,x') \,\, \frac{d(x,x')+d(x,x^*)}{2t}\\
&=&d(x^*,x') \left(L- \frac{2d(x,x^*)-d(x^*,x')}{2t}\right)\\
&\leq& d(x^*,x')\left(L- \frac{2(1+\e)tL-d(x^*,x')}{2t}\right)\\
&=&d(x^*, x')\left(\frac{d(x^*,x')-2\e tL}{2t}\right)\;\;=\;\;-\frac12 \e^2 tL.
\end{eqnarray*}
So $F(x')<F(x^*)$ which contradicts the minimality of $x^*$, and completes the proof.
\end{proof}

Part (i) below, which shows that the proximal map is \emph{nonexpansive}, is well-known \cites{bacak2014}, and part (ii) follows from \cite[Prop. 3.7]{azagra-ferrera2006}. Recall that
$
d_x(\,\cdot\,):=d(x,\,\cdot\,).
$
 
 \begin{lemma}[\cites{bacak2014,azagra-ferrera2006}]\label{lem:prox} 
 Let $u$ be a convex function  on a Cartan-Hadamard manifold. Then
 \begin{enumerate}[(i)]
\item{$d(x^*_1, x^*_2) \leq d(x_1,x_2),$}
 \item{If $x^*$ is a regular point of $u$, then
 $$
 \nabla u(x^*)=-\frac{d(x,x^*)}t \nabla d_x(x^*),\quad\quad\text{and} \quad\quad
 \nabla \tilde u^{\,t}(x)=\frac{d(x,x^*)}t \nabla d_{x^*}(x).
$$ }
  \item{$\nabla u (x^*)$ and $\nabla \tilde u^{\,t} (x)$ are tangent  to the geodesic connecting  $x^*$ to $x$, and 
  $$
  |\nabla u (x^*)|=|\nabla \tilde u^{\,t} (x)|.
  $$}
  \item{If $u$ is $L$-Lipschitz, then so is $\tilde u^{\,t}$.}
 \end{enumerate}
 \end{lemma}
\begin{proof} 
For part (i)  see  \cite[Thm. 2.2.22]{bacak2014}. For part (ii) note that by definition $F(y)\geq F(x^*)$. Furthermore,   $x^*$ is a regular point of $F$, since by assumption $x^*$ is a regular point of $u$. Consequently, 
 $$
0= \nabla F(x^*)= \nabla_y F(x,y)\big|_{y=x^*}=\nabla u(x^*)+\frac{d(x,x^*)}t \nabla d_x(x^*),
 $$
which yields the first equality in (ii). 
Next we prove the second inequality in (ii)  following \cite[Prop. 3.7]{azagra-ferrera2006}. 
To this end note that 
\begin{eqnarray*}
\tilde u^{\,t}(z)&=& \inf_y F(z,y)\leq F(z,x^*)= u(x^*)+\frac{d^2(z,x^*)}{2t},\\
\tilde u^{\,t}(x)&=&  \inf_y F(x,y)=F(x,x^*)=u(x^*)+ \frac{d^2(x,x^*)}{2t}.
\end{eqnarray*}
So it follows that
$$
\tilde u^{\,t}(z)-\frac{d(z,x^*)^2}{2t}\leq u(x^*)=\tilde u^{\,t}(x)-\frac{d^2(x,x^*)}{2t}.
$$
Hence $g(\cdot):=\tilde u^{\,t}(\cdot)-d(\cdot,x^*)^2/(2t)$ achieves its maximum at $x$. Further note that $g$ is $\C^1$ since
$\tilde u^{\,t} $ is $\C^1$  by Lemma \ref{lem:infconv0}. Thus 
$$
0=\nabla g(x)=  \nabla \tilde u^{\,t}(x)-\frac{d(x,x^*)}t \nabla d_{x^*}(x),
$$
which yields the second equality in (ii). Next, to establish (iii), let $\alpha\colon[0, s_0]\to M$ be the  geodesic with $\alpha(0)=x^*$ and $\alpha (s_0)=x$. Then, by Lemma \ref{lem:CS}, 
$$
\nabla d_x(x^*)=-\alpha'(0),\quad\quad\text{and}\quad\quad\nabla d_{x^*}(x)=\alpha'(s_0).
$$
So $\nabla u(x^*)$ and $\nabla \tilde u^t(x)$ are tangent to $\alpha$ and 
$$
|\nabla u(x^*)|=\left|\frac{d(x,x^*)}{t}\right|=|\nabla \tilde u^t(x)|
$$ 
as desired. Finally, to establish (iv), note that if $u$ is $L$-Lipschitz, then $|\nabla u|\leq L$ almost everywhere. Thus by part (iii), $|\nabla\tilde u^t(x)|=|\nabla u(x^*)|\leq L$ for almost every $x\in M$. So $\tilde u^t$ is $L$-Lipschitz.
\end{proof}

Recall that we say  a function $u\colon M\to\R$ is \emph{locally $\C^{1,1}$} provided that it is $\C^{1,1}$ in some choice of local coordinates around each point.  There are other notions of $\C^{1,1}$ regularity  \cites{fathi2003,azagra-ferrera2015} devised in order to control  the Lipschitz constant; however, all these definitions yield the same class of locally $\C^{1,1}$ functions; see \cite{azagra-ferrera2015}.
The $\C^{1,1}$ regularity of functions is closely related to the more robust notion of semiconcavity which is defined as follows. We say that $u$ is \emph{$C$-semiconcave}  (or is uniformly semiconcave with a constant $C$) on a set $\Omega\subset M$ provided that there exists a constant $C>0$ such that for every $x_0\in \Omega$, the function
\be\label{eq:f}
x\;\mapsto\; u(x)-Cd^2(x,x_0)
\ee
 is concave on $\Omega$. Furthermore, we say $u$ is \emph{$C$-semiconvex}, if $-u$ is semiconcave.

\begin{lemma}[\cites{cannarsa-sinestrari2004,azagra-ferrera2015}]\label{lem:convex-concave}
If a function $u$ on a Riemannian manifold is  both $C/2$-semiconvex and $C/2$-semiconcave on some bounded domain $\Omega$, then  it is locally $\C^{1,1}$ on $\Omega$. Furthermore $|\nabla^2 u|\leq C$ almost everywhere on $\Omega$.
\end{lemma}
The above fact is well-known in $\R^n$, see \cite[Cor. 3.3.8]{cannarsa-sinestrari2004} (note that the constant $C$ in the book of Cannarsa and Sinestrari \cite{cannarsa-sinestrari2004}  corresponds  to $2C$ in this work due to a factor of $1/2$ in their definition of semiconcavity.) The Riemannian analogue follows from the Euclidean case via local coordinates to obtain the $\C^{1,1}$ regularity (since semiconcavity is preserved under $\C^2$ diffeomorphisms),  and then differentiating along geodesics to estimate the Hessian, see the proof of \cite[Cor. 3.3.8]{cannarsa-sinestrari2004}, and using Rademacher's theorem. The above lemma has also been established in
 \cite[Thm 1.5]{azagra-ferrera2015}. The next observation, with a nonexplicit estimate for $C$, has been known \cite[Prop. 7.1(2)]{azagra-ferrera2015}. Here we provide another argument via Lemma \ref{lem:infconv2}.

\begin{proposition}\label{prop:C/t}
Suppose that $u$ is a convex function on a bounded domain $\Omega$ in a Riemannian manifold. Then for all $0<t\leq t_0$, $\tilde u^{\,t}$ is  $C/(2t)$-semiconcave on $\Omega$
for
\be\label{eq:C}
C\geq\sqrt{-K_0}\,3t_0L\, \coth\left(\sqrt{-K_0}\, 3t_0L\right), 
\ee
where $K_0$ is the lower bound for the curvature of $B_{t_0L}(\Omega)$, and $L$ is the Lipschitz constant of $u$ on $\Omega$.
 In particular, $\tilde u^{\,t}$ is locally $\C^{1,1}$, and
\be\label{eq:C/t}
|\nabla^2 \tilde u^{\,t}|\leq \frac{C}{t}
\ee
almost everywhere  on $\Omega$.
\end{proposition}

\begin{proof}
Since by Lemma \ref{lem:infconv0}, $\tilde u^{\,t}$ is convex,  it is $C/(2t)$-semiconvex. Thus as soon as we show that $\tilde u^{\,t}$ is $C/(2t)$-semiconcave, $\tilde u^{\,t}$ will be $\C^{1,1}$ and \eqref{eq:C/t} will hold by Lemma \ref{lem:convex-concave}, which will finish the proof. 
To establish the semiconcavity of $\tilde u^{\,t}$ note that, by Lemma \ref{lem:infconv2},
$$
\tilde u^{\,t}(x)=\inf_{y\in B_{tL}(x)}\left(u(y)+\frac1{2t}d^2(x,y)\right).
$$
Let $C$ be as in \eqref{eq:C} and, according to \eqref{eq:f}, set
$$
f(x):=\tilde u^{\,t}(x)-\frac{C}{2t}d^2(x,x_0)=\inf_{y\in B_{tL}(x)}\left(u(y)- \frac1{2t}\Big(Cd^2(x,x_0)-d^2(x,y)\Big)\right).
$$
We have  to show that $f$ is concave on $\Omega$. 
To this end it suffices to show that $f$ is locally concave on $\Omega$, since a locally concave function is concave.
Indeed suppose that $f$ is locally concave on $\Omega$ and let $\alpha\colon[a,b]\to\Omega$ be a  geodesic. Then $-f\circ\alpha$ is locally convex. Thus, since $-f\circ\alpha$ is $\C^1$, $-(f\circ\alpha)'$ is nondecreasing, which yields that 
$-f\circ\alpha$ is convex \cite[Thm. 1.5.10]{schneider2014}.
Now, to establish that $f$ is locally concave on $\Omega$, set 
$$
r:=t_0L.
$$
We claim that $f$ is concave on $B_r(p)$, for all $p\in\Omega$. To see this first note that if $x\in B_{r}(p)$ then $B_{tL}(x)\subset B_{r}(x)\subset B_{2r}(p)$. So, for $x\in B_{r}(p)$,
\be\label{eq:C/2t}
f(x)=\inf_{y\in B_{2r}(p)}\left(u(y)- \frac1{2t}\Big(Cd^2(x,x_0)-d^2(x,y)\Big)\right).
 \ee
Since the infimum of a family of concave functions is concave, it suffices to check that the functions on the right hand side of \eqref{eq:C/2t} are concave on $B_{2r}(p)$ for each $y$. So we need to show  that 
$$
g(x):=\frac{1}{2}\big(Cd^2(x,x_0)- d^2(x,y)\big)
$$
is convex on $B_{2r}(p)$ for each $y$. To this end note that the  eigenvalues of $\nabla^2d^2_{x_0}(x)/2$ are bounded below by $1$ \cite[Thm. 6.6.1]{jost2017}. Furthermore, since $x\in B_{r}(p)$, and $y\in B_{2r}(p)$, we have $x\in B_{3r}(y)$.
Thus the eigenvalues of $\nabla ^2 d^2_y(x)/2$  are bounded above by
$$
\lambda:=\sqrt{-K_0}\,3r \coth\left(\sqrt{-K_0}\, 3r\right),
$$
by \cite[Thm. 6.6.1]{jost2017}.
So the eigenvalues of $\nabla^2g$ on $B_{2r}(p)$ are bounded below by $C-\lambda$. Hence $g$ is convex on $B_{2r}(p)$ if $C\geq \lambda$, which is indeed the case by \eqref{eq:C}. So $f$ is concave on $B_{2r}(p)$ which completes the proof.
\end{proof}

\begin{proposition}\label{prop:infconv} 
Let $\Gamma$ be a closed hypersurface in a Cartan-Hadamard manifold $M$ and set $u:=\widehat{d}_\Gamma$. Then 
\begin{enumerate}[(i)]
\item{$\tilde u^{\,t}=u-t/2$ on $M \setminus U_t(\textup{cut}(\Gamma))$.}
\item{$|\nabla \tilde u^{\,t}|\equiv1$ on $M\setminus U_t(\textup{cut}(\Gamma))$}.
\item{$|\nabla \tilde u^{\,t}|\leq1$ on $M$ if $\Gamma$ is $d$-convex.}
\end{enumerate}
\end{proposition}
\begin{proof}
Let $x\in M\setminus \cl(U_t(\textup{cut}(\Gamma)))$. Then $x^*$ is a regular point of $u$ by Lemma \ref{lem:infconv2}. So
$$
0=\nabla F(x^*)=\nabla u(x^*) +\frac{d(x,x^*)}{t} \nabla d_x(x^*).
$$
Since $|\nabla u(x^*)|=|\nabla d_x(x^*)|=1$, it follows that 
\be\label{eq:d-x-x*}
d(x,x^*)=t.
\ee
Furthermore, we obtain $\nabla u(x^*) =-\nabla d_x(x^*)$. But $\nabla d_x(x^*)$ is tangent to the geodesic which passes through $x^*$ and $x$, while $\nabla u(x^*)$ is tangent to the flow line of $\nabla u$ through $x^*$, which is also a geodesics. So $x^*$ lies on the geodesic $\alpha$, given by $\alpha(0)=x$ and $\alpha'(0)=\nabla u(x)$. Note that
$
u(\alpha(t))=u(x)+t.
$
Furthermore, by  \eqref{eq:d-x-x*}, either $x^*=\alpha(t)$ or $x^*=\alpha(-t)$. If $x^*=\alpha(-t)$, then 
$$
u(x^*)=u(\alpha(-t))=u(x)-t.
$$
Hence
$$
\tilde u^t(x)=F(x^*)=u(x^*)+\frac{t^2}{2t}=u(x)-\frac{t}{2},
$$
as desired. If on the other hand $x^*=\alpha(t)$, then a similar computation yields that $\tilde u^t(x)=u(x)+t/2>u(x)$, which is not possible.
So we have established part (i) of the proposition. To see part (ii) note that $|\nabla u|\equiv1$ on 
$M\setminus U_t(\textup{cut}(\Gamma))$. Thus by (i) $|\nabla \tilde u^t|\equiv|\nabla u|\equiv1$ on $M\setminus U_t(\textup{cut}(\Gamma))$. Finally, part (iii) follows immediately from part (iv) of Lemma \ref{lem:prox}.
\end{proof}

\section{Cut Locus of Convex Hypersurfaces}\label{sec:projection}
Recall that a hypersurface  is $d$-convex if its distance function is convex, as we discussed in Section \ref{sec:d-convex}. Here 
we will study the cut locus of $d$-convex hypersurfaces and establish the following result:

\begin{theorem}\label{thm:projection}
Let $\Gamma$ be a $d$-convex hypersurface in a Cartan-Hadamard manifold $M$, and let $\Omega$ be the convex domain bounded by $\Gamma$.  Then
for any point $x\in\Omega$ and any of its footprints $x^\circ$  in $\textup{cut}(\Gamma)$, 
$$
d_\Gamma(x^\circ) \geq d_\Gamma(x).
$$
\end{theorem}

Throughout this section we will assume that $M$ is a Cartan-Hadamard manifold. In particular the exponential map $\exp_p\colon T_p M\to\R^n$ will be a global diffeomorphism.
The proof of the above theorem is based on the notion of tangent cones, which we defined in Section \ref{sec:CH}. Another approach to proving this result is discussed in a recent work of Kapovitch and Lytchak \cite{kapovitch-lytchak}. 
We start by recording that  for a given a set $X\subset\R^n$ and $p\in X$, $T_p X$ is the limit of dilations of $X$ based at $p$   \cite[Sec. 2]{ghomi-howard2014}. More precisely, if we identify $p$ with the origin $o$ of $\R^n$, and
for $\lambda\geq 1$ set
 $
\lambda X:=\{\lambda x\mid x\in X\},
 $
then $T_o X$ is the \emph{outer limit} \cite{rockafellar-wets1998} of the sets $\lambda X$:
\be\label{eq:outer}
 T_o X=\limsup_{\lambda\to\infty} \lambda X.
\ee
This means that for every $x\in T_o X\setminus\{o\}$ there exists a sequence of numbers $\lambda_i\to\infty$ such that $\lambda_iX$ eventually intersects any open neighborhood of $x$. Equivalently, we may record that:

\begin{lemma}[\cite{ghomi-howard2014}]\label{lem:ToX}
Let $X\subset\R^n$ and $o\in X$. Then $x\in T_o X\setminus\{o\}$ if there exists a sequence of points $x_i\in X\setminus\{o\}$ such that  $x_i\to o$ and $x_i/| x_i|\to x/|x|$. 
\end{lemma}

The last lemma yields:

  \begin{lemma}\label{lem:GammaR^n}
Let $\Gamma\subset \R^n$ be a closed hypersurface, and $o\in\textup{cut}(\Gamma)\cap \Gamma$. Suppose that $T_o\Gamma$ bounds a convex cone containing $\Gamma$. Then
$$
\textup{cut}(T_{o}\Gamma)\subset T_{o}\textup{cut}(\Gamma).
$$
\end{lemma}
  \begin{proof}
By  \eqref{eq:medial} $\textup{cut}(T_{o}\Gamma)=\cl(\text{medial}(T_o\Gamma))$. So it suffices to show that $\text{medial}(T_o\Gamma)\subset T_{o}\textup{cut}(\Gamma)$, since $\textup{cut}(T_{o}\Gamma)$ is closed by definition.  Let $x\in \text{medial}(T_o\Gamma)$. Then there exists a sphere $S$ centered at $x$ which is contained in (the cone bounded by) $T_{o}\Gamma$, and touches $T_{o}\Gamma$ at multiple points.  Suppose that $S$ has radius $r$. Then, by \eqref{eq:outer}, for each natural number $i$ we may choose a number $\lambda_i$ so large that the sphere $S_i$ of radius $r-(1/i)$ centered at $x$ is contained in ${\lambda_i}\Gamma$. Let $S_i'$ be the largest sphere contained in ${\lambda_i}\Gamma$ centered at $x$ which contains $S_i$. Then $S_i'$ must intersect ${\lambda_i}\Gamma$ at some point $y$. Let $S_i''$ be the largest sphere contained in $\lambda_i\Gamma$ which passes through $y$.  Then the center $c_i$ of $S_i''$ lies in $\textup{skeleton}(\lambda_i\Omega)$, and therefore belongs to
$\textup{cut}({\lambda_i}\Gamma)$, by Lemma \ref{lem:skeleton}. Now note that $ \textup{cut}({\lambda_i}\Gamma)=\lambda_i\textup{cut}(\Gamma)$. 
So  
$$
x_i:=\frac{c_i}{\lambda_i}\in \frac{\textup{cut}({\lambda_i}\Gamma)}{\lambda_i}\in\textup{cut}(\Gamma).
$$
  Furthermore, note that $c_i\to x$, since $S_i''$ and $S_i$ have a point in common, $S_i''$ is a maximal sphere in $\lambda_i\Gamma$, $S_i$ is a maximal sphere in $T_o\Gamma$, and $\lambda_i\Gamma\to T_o\Gamma$ according to \eqref{eq:outer}. Thus $x_i\to o$, and $x_i/|x_i|\to x/|x|$. So $x\in T_{o}\textup{cut}(\Gamma)$ by Lemma \ref{lem:ToX}, which completes the proof.
 \end{proof}

For any set $X\subset\R^n$ we define $\cone(X)$ as the set of all rays which emanate from the origin $o$ of $\R^n$ and pass through a point of $X$. Furthermore we set 
$$
\arc{X}:= X\cap\S^{n-1}.
$$

\begin{lemma}\label{lem:tildeSC}
Let $X$ be the boundary of a proper convex cone  $C$ with interior points in $\R^n$ and apex at $o$. Suppose that $X$ is not a hyperplane.
 Then
$$
\arc{\textup{cut}(X)}=\textup{cut}(\arc{X}),
$$
where $\textup{cut}(\arc{X})$ denotes the portion of the cut locus of $\arc X$ as a hypersurface in $\S^{n-1}$, which is contained in $C$.
\end{lemma}
\begin{proof}
Let $x\in \arc{\textup{cut}(X)}$. Then, since $X$ is not a hyperplane,  there exists a sphere $S$ centered at $x$ which is contained inside the cone bounded by $X$  and touches $X$ at multiple points, or else $x$ is a limit of the centers of such spheres, by \eqref{eq:medial}. Consequently, $\arc{\cone(S)}$ forms a sphere in $\S^{n-1}$, centered at $x$, which is contained inside $\arc X$ and touches $\arc X$ at multiple points, or is the limit of such spheres respectively. Thus $x$ belongs to $\textup{cut}(\arc X)$, which yields that
$
\arc{\textup{cut}( X)}\subset\textup{cut}(\arc X).
$
The reverse inequality may be established similarly.
\end{proof}

Using the last lemma, we next show:

\begin{lemma}\label{lem:cones}
Let $X$ be as in Lemma \ref{lem:tildeSC}. Suppose that $X$ is not a hyperplane. Then for every point $x\in X$, there exists a point $s\in \textup{cut}(X)$ such that 
$$
\langle s, x\rangle >0.
$$
\end{lemma}
\begin{proof}
We may replace $x$ by $x/|x|$. Then, by Lemma \ref{lem:tildeSC}, it is enough to show that $\langle s,x\rangle > 0$ for some $s\in \textup{cut}(\arc X)$, or equivalently that $\delta_{\S^{n-1}}(s,x)< \pi/2$, where $\delta_{\S^{n-1}}$ denotes the distance in $\S^{n-1}$. To this end let $s$ be a footprint of $x$ on $\textup{cut}(\arc X)$. Suppose towards a contradiction that $\delta_{\S^{n-1}}(s,x)\geq \pi/2$. Consider the great sphere $G$ in $\S^{n-1}$ which passes through $s$ and is orthogonal to the geodesic segment $xs$; see Figure \ref{fig:sphere}. Let $G^+$ be the hemisphere bounded by $G$ which contains $x$. Then the interior of $G^+$ is disjoint from $\textup{cut}(\arc X)$, since $\delta_{\S^{n-1}}(x,G)\leq \delta_{\S^{n-1}}(x,s)=\delta_{\S^{n-1}}(x,\textup{cut}(\arc X))$.
Next note that the intersection of the convex cone bounded by $X$ with $\S^{n-1}$ is a convex set in $\S^{n-1}$. Thus $G$ divides this  convex set  into two subregions. Consider the region, say $R$, which contains $x$, or lies in $G^+$, and let $S$ be a sphere of largest radius in $R$. Then $S$ must touch the boundary of $R$ at least twice. Since $S$ cannot touch $G$ more than once, it follows that $S$ must touch $\arc X$, because the boundary of $R$ consists of a part of $G$ and a part of $\arc X$. First suppose that  $S$ touches $\arc X$ multiple times. Then the center of $S$ belongs to $\textup{cut}(\arc X)$. But this is impossible, since $S\subset R\subset G^+$. We may suppose then that $S$ touches $\arc X$ only once, say at a point $y$. 

\begin{figure}[h]
\centering
\begin{overpic}[height=1.5in]{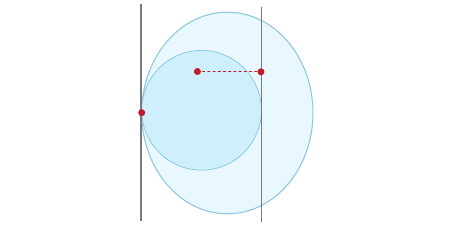}
\put(54,-1.5){\Small $G$}
\put(26,-1){\Small $G'$}
\put(27.5,24){\Small $y$}
\put(40,33){\Small $x$}
\put(59.5,33){\Small $s$}
\put(36,18){\Small $S$}
\put(68,10){\Small $\arc X$}
\end{overpic}
\caption{}\label{fig:sphere}
\end{figure}
Now we claim that the diameter of $S$ is $\geq \pi/2$. Indeed let $G'$ be the great sphere which passes through $y$ and is tangent to $S$. Then $G'$ supports $\arc X$, and $R$ is contained entirely between $G$ and $G'$. The maximum length of a geodesic segment orthogonal to both $G$ and $G'$ is then equal to the diameter of $S$, since the points where $S$ touches $G$ and $G'$ must be antipodal points of $S$. In particular the length of the diameter of $S$ must be greater than $\delta_{\S^{n-1}}(x,s)$ as desired.

Finally let $S'$ be the largest sphere contained in $\arc X$ which passes through $y$. Then the center, say $z$, of $S'$ belongs to $\textup{cut}(\arc X)$ by Lemma \ref{lem:skeleton}. But the diameter of $S'$ is  $< \pi$, since $X$ is not a hyperplane by assumption. So, 
since the diameter of $S$ is $\geq\pi/2$, it follows that $z$ is contained in the interior of $S$ and therefore in the interior $R$. Hence we reach the desired contradiction since, as we had noted earlier, $R$ does not contain points of $\textup{cut}(\arc X)$ in its interior.
\end{proof}

For $x\in\Omega$, set 
 $$
 \Omega_x:=\big\{y\in\Omega\mid d_\Gamma(y)> d_\Gamma(x)\big\},\quad\quad\text{and}\quad\quad \Gamma_x:=\partial\Omega_x.
 $$

\begin{lemma}\label{lem:SGamma_x}
$
\textup{cut}(\Gamma_{x})\subset \textup{cut}(\Gamma).
$
\end{lemma}
\begin{proof}
As we discussed in the proof of Lemma \ref{lem:GammaR^n}, 
it suffices to show that $\text{medial}(\Gamma_x)\subset \textup{cut}(\Gamma)$ by \eqref{eq:medial} .
Let $y\in \text{medial}(\Gamma_x)$. Then there exists a sphere $S\subset\cl(\Omega_x)$ centered at $y$ which intersects $\Gamma_x$ in multiple points. Let $S'$ be the sphere centered at $y$ with radius equal to the radius of $S$ plus $d(x,\Gamma)$. Then $S'\subset\cl(\Omega)$ and it intersects $\Gamma$ in multiple points. So, again by  \eqref{eq:medial}, $y\in \textup{cut}(\Gamma)$ as desired. 
\end{proof}

 We need to record  one more observation, before proving Theorem \ref{thm:projection}. An example of the phenomenon stated in the following lemma occurs when $\Gamma$ is the inner parallel curve of a (noncircular) ellipse in $\R^2$ which passes through the foci of the ellipse, and $p$ is one of the foci. 
 
 \begin{lemma}\label{lem:orthogonal-ray}
 Let $\Gamma$ be a $d$-convex hypersurface in a Cartan-Hadamard manifold $M$, and $p\in\Gamma\cap\textup{cut}(\Gamma)$. Suppose that $T_p\Gamma$ is a hyperplane. Then $T_p \textup{cut}(\Gamma)$ contains a ray which is orthogonal to $T_p\Gamma$.
 \end{lemma}
  \begin{proof}
 Let  $\alpha(t)$, $t\geq 0$, be the geodesic ray, with $\alpha(0)=p$, such that $\alpha'(0)$ is orthogonal to $T_p \Gamma$ and points towards $\Omega$. We have to show that $\alpha'(0)\in T_p\textup{cut}(\Gamma)$. To this end we divide the argument into two cases as follows.
 
 First suppose that there exists a sphere in $\cl(\Omega)$ which touches $\Gamma$ only at $p$. Then the center of that sphere coincides with $\alpha(t_0)$ for some $t_0>0$. We claim that then $\alpha(t)\in \textup{cut}(\Gamma)$ for all $t\leq t_0$. To see this note that $\alpha(t)$ has a unique 
 footprint on $\Gamma$, namely $p$, for all $t\leq t_0$. For $0<t\leq t_0$, let
$\Gamma^{t}:=(\widehat{d}_\Gamma)^{-1}(-t)$ be the inner parallel hypersurface of $\Gamma$ at distance $t$. Suppose, towards a contradiction, that $\alpha(t)\not\in \textup{cut}(\Gamma)$. Then, by Lemma \ref{lem:CS},  $\widehat{d}_\Gamma$ is $\C^1$ near $\alpha(t)$, which in turn yields that $\Gamma^t$ is $\C^1$ in a neighborhood $U^t$ of $\alpha(t)$. Furthermore, 
$\Gamma^t$ is convex by the $d$-convexity assumption on $\Gamma$. So, by Lemma \ref{lem:outward-geodesic}, the outward geodesic rays which are perpendicular to $U^t$ never intersect, and thus yield a homeomorphism between $U^t$ and a neighborhood $U$ of $p$ in $\Gamma$. Furthermore,  since $\widehat{d}_\Gamma$ is $\C^1$ near $U^t$, each point of $U^t$ has a unique footprint on $\Gamma$  by Lemma \ref{lem:CS}. Thus there exists a sphere centered at each point of $U^t$ which lies in $\cl(\Omega)$ and passes through a point of $U$. Furthermore each point of $U$ is covered by such a sphere. So it follows that a ball rolls freely on the convex side of $U$, and therefore $U$ is $\C^{1,1}$, by the same argument we gave in the proof of Lemma \ref{lem:GH}. But, again by Lemma \ref{lem:GH}, if $U$ is $\C^{1,1}$, then $\widehat{d}_\Gamma$ is $\C^1$ near $U$, which is not possible since $p\in U$ and $p\in\textup{cut}(\Gamma)$. Thus we arrive at the desired contradiction. So we conclude that $\alpha(t)\in \textup{cut}(\Gamma)$ as claimed, for $0<t\leq t_0$, which in turn yields that $\alpha'(0)\in T_p\textup{cut}(\Gamma)$ as desired.

So we may assume that there exists no sphere in $\cl(\Omega)$ which touches $\Gamma$ only at $p$. Now for small $\epsilon>0$ let $S_\e $ be a sphere of radius $\e$ in $\cl(\Omega)$ whose center $c_\e $ is as close to $p$ as possible, among all spheres of radius $\e$ in $\cl(\Omega)$. Then $S_\e $ must intersect $\Gamma$ in multiple points, since $\Gamma$ is convex and $S_\e $ cannot intersect $\Gamma$ only at $p$. Thus $c_\e \in\textup{cut}(\Gamma)$.  Let $v$ be the initial velocity of the geodesic $c_\e p$, and 
$\theta(\e)$ be the supremum of the angles between $v$ and the initial velocities of the geodesics connecting $c_\e $ to each of its footprints on $\Gamma$. We claim that $\theta(\e)\to 0$, as $\e\to 0$. To see this let $(T_{c_\e}M)^1$ denote the unit sphere in $T_{c_\e }M$, centered at $c_\e$.  Furthermore, let $X\subset(T_{c_\e}M)^1$ denote the convex hull spanned by the initial velocities of the geodesics connecting $c(\e)$ to its footprints. Then $v$ must lie in  $X$, for otherwise $S_\e$ may be pulled closer to $p$. Indeed if $v\not\in X$, then $v$ is disjoint from a closed hemisphere of $(T_{c_\e}M)^1$ containing $X$.  Let $w$ be the center of the opposite hemisphere. Then $\langle v, w\rangle>0$. Thus perturbing $c(\e)$ in the direction of $w$ will  bring $S_\e$ closer to $p$ without leaving $\cl(\Omega)$, which is not possible. So $v\in X$ as claimed. Now note that  the footprints of $c_\e $ converge to $p$, since $c_\e $ converges to $p$. Furthermore, since $T_p\Gamma$ is a hyperplane, it follows that the angle between every pair of geodesics which connect $c_\e$ to its footprints vanishes. Thus $X$ collapses to a single point, which can only be $v$. Hence $\theta(\e)\to 0$ as claimed. Consequently $c_\e p$ becomes arbitrarily close to meeting $\Gamma$ orthogonally, or more precisely, the angle  between $\alpha'(0)$ and  the initial velocity vector of $pc_\e$  vanishes as $\e\to 0$. Hence,  since $c_\e \in\textup{cut}(\Gamma)$, it follows once again that $\alpha'(0)\in T_p\textup{cut}(\Gamma)$ which completes the proof.
\end{proof}

\begin{figure}[h]
\centering
\begin{overpic}[height=1.5in]{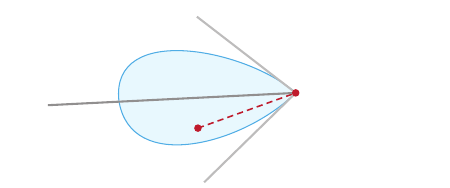}
\put(67.5,20.5){\Small $x^\circ$}
\put(40.5,12.5){\Small $x$}
\put(27,30){\Small $\Gamma_{x^\circ}$}
\put(50,3){\Small $T_{x^\circ}\Gamma_{x^\circ}$}
\put(4,20){\Small $T_{x^\circ}\textup{cut}(\Gamma_{x^\circ})$}
\put(35,24){\Small$\Omega_{x^\circ}$}
\end{overpic}
\caption{}\label{fig:egg}
\end{figure}

Finally we are ready to prove the main result of this section:

\begin{proof}[Proof of Theorem \ref{thm:projection}]
 Suppose, towards a contradiction, that $d(x,\Gamma)>d(x^\circ,\Gamma)$ for some point $x\in\Omega$. Then 
\be\label{eq:intOmegax}
x\in \Omega_{x^\circ},
\ee
see Figure \ref{fig:egg}. 
Since $x^\circ$ is a footprint of $x$ on $\Gamma$,  $\textup{cut}(\Gamma)$ lies outside a sphere of radius $d(x^\circ,x)$ centered at $x$.  So if we let $v$ be the initial velocity of the geodesic $x^\circ x$, then
$
\langle y, v\rangle\leq 0,
$
for all $y\in T_{x^\circ}\textup{cut}(\Gamma)$, where we identify $T_{x^\circ}\textup{cut}(\Gamma)$ with $\R^n$ and $x^\circ$ with the origin of $\R^n$.
By Lemma \ref{lem:SGamma_x},  
$
T_{x^\circ}\textup{cut}(\Gamma_{x^\circ})\subset T_{x^\circ}\textup{cut}(\Gamma).
$
 Thus 
 $
\langle y, v\rangle\leq 0,
$
for all $y\in T_{x^\circ}\textup{cut}(\Gamma_{x^\circ})$. Furthermore, by Lemma \ref{lem:GammaR^n},
$
\textup{cut}(T_{x^\circ}\Gamma_{x^\circ})\subset T_{x^\circ}\textup{cut}(\Gamma_{x^\circ}).
$
So 
\be\label{eq:s-tildex}
\langle s, v\rangle\leq 0, \quad\text{for all}\quad s\in \textup{cut}(T_{x^\circ}\Gamma_{x^\circ}).
\ee
Furthermore, since $\Gamma$ is $d$-convex, $T_{x^\circ}\Gamma_{x^\circ}$ bounds a convex cone by Lemma \ref{lem:CG}. Thus, since $T_{x^\circ}\Gamma_{x^\circ}$   contains $v$,
it must be a hyperplane, by Lemma \ref{lem:cones}. Consequently, by Lemma \ref{lem:orthogonal-ray}, $T_{x^\circ}\textup{cut}(\Gamma_{x^\circ})$ contains a ray which is orthogonal to $T_{x^\circ}\Gamma_{x^\circ}$. 
By \eqref{eq:s-tildex}, $v$ must be orthogonal to that ray. So $v\in T_{x^\circ}\Gamma_{x^\circ}$, which in turn yields that $x\in\Gamma_{x^\circ}$. The latter is impossible by \eqref{eq:intOmegax}. Hence we arrive at the desired contradiction.
\end{proof}

Having established Theorem \ref{thm:projection}, we record the following consequence of it. Set
$$
\widehat r(\,\cdot\,):=d\big(\,\cdot\,, \textup{cut}(\Gamma)\big).
$$
Recall that, by Lemma \ref{lem:lipschitz}, $\widehat r$ is Lipschitz and thus is differentiable almost everywhere.
\begin{corollary}\label{cor:projection}
Let $\Gamma$ be a $d$-convex hypersurface in a Cartan-Hadamard manifold $M$, and set $u:=\widehat{d}_\Gamma$. 
Suppose that $\widehat{r}$ is differentiable at a point $x\in M\setminus\textup{cut}(\Gamma)$.
Then
$$
\big\langle \nabla u(x),\nabla\widehat{r}(x)\big\rangle\geq 0.
$$
In particular (since $\widehat{r}$ is Lipschitz), the above inequality holds for almost every $x\in M\setminus\textup{cut}(\Gamma)$.
\end{corollary}
\begin{proof}
Since $\widehat r$ is differentiable at $x$, $x$ has a unique footprint $x^\circ$ on $\textup{cut}(\Gamma)$, by Lemma \ref{lem:CS}(i).
Let $\alpha$ be a  geodesic connecting $x$ to $x^\circ$. Then, by Lemma \ref{lem:CS}(ii), $\alpha'(0)=-\nabla\widehat{r}(x)$. Furthermore, by Theorem \ref{thm:projection}, $u\circ\alpha=-\widehat{d}_\Gamma\circ\alpha$ is nonincreasing. Finally, recall that by Proposition \ref{prop:C11}, $u$ is $\C^1$ on $M\setminus\textup{cut}(\Gamma)$, and therefore $u\circ\alpha$ is $\C^1$ as well. Thus
$$
0\geq (u\circ\alpha)'(0)=\big\langle \nabla u(\alpha(0)), \alpha'(0)\big\rangle=\big\langle \nabla u(x), -\nabla\widehat{r}(x)\big\rangle,
$$
as desired.
\end{proof}

\addtocontents{toc}{\protect\setcounter{tocdepth}{0}}
\section*{Acknowledgments}

M.G. would like to thank Andrzej \'{S}wi\k{e}ch  for helpful conversations on various aspects of this work. Thanks also to Daniel Azagra, Igor Belegradek, Albert Fathi, Robert Greene, Ralph Howard, Sergei Ivanov, Bruce Kleiner, Alexander Lytchak, Daniil Mamaev, Anya Nordskova, Anton Petrunin, Manuel Ritore, Rolf Schneider, and Yao Yao for useful communications. Finally we thank the anonymous referee for comments which led to  improvements in this work.

\addtocontents{toc}{\protect\setcounter{tocdepth}{1}}
\bibliography{references}

\end{document}